\documentclass{amsart}

\usepackage{floatrow}
\usepackage{caption}
 \usepackage[foot]{amsaddr}
\let\oldtocsection=\tocsection
 
\let\oldtocsubsection=\tocsubsection 
 
\let\oldtocsubsubsection=\tocsubsubsection

\renewcommand{\tocsection}[2]{\vspace{0.5em}\hspace{0em}\oldtocsection{#1}{#2}}
\renewcommand{\tocsubsection}[2]{\vspace{0.5em}\hspace{1em}\oldtocsubsection{#1}{#2}}
\renewcommand{\tocsubsubsection}[2]{\vspace{0.5em}\hspace{2em}\oldtocsubsubsection{#1}{#2}}
\usepackage{graphicx,cancel,xcolor,hyperref,comment,graphicx,geometry}
 
\usepackage{amsopn}

\DeclareMathOperator{\supp}{supp}
\DeclareMathOperator{\divv}{div}
\DeclareMathOperator{\meas}{meas}

\usepackage{tikz}
\usepackage{graphicx,url,etoolbox}
\usepackage{lipsum}
\usepackage{dsfont}
 \usepackage{amssymb}

\usepackage{graphicx,cancel,xcolor,hyperref,comment,graphicx,geometry}

\usepackage{tikz}
\usetikzlibrary{decorations.pathreplacing}
\usepackage{graphicx,url,etoolbox}
\usepackage{lipsum}
\usepackage{dsfont}

\usepackage{fancyhdr}
\usepackage{lastpage}

\newtheorem{theoreme}{Theorem}[section]

\newtheorem{lemma}[theoreme]{Lemma}
\newtheorem{rem}[theoreme]{Remark}

\newtheorem{definition}[theoreme]{Definition}
\theoremstyle{definition}

\numberwithin{equation}{section}

 \renewenvironment{proof}{{\bfseries \noindent Proof.}}{\demo}
\newcommand\xqed[1]{%
  \leavevmode\unskip\penalty9999 \hbox{}\nobreak\hfill
  \quad\hbox{#1}}
\newcommand\demo{\xqed{$\square$}}

\hypersetup{bookmarks, colorlinks, urlcolor=blue, citecolor=blue, linkcolor=blue, hyperfigures, pagebackref,
    pdfcreator=LaTeX, breaklinks=true, pdfpagelayout=SinglePage, bookmarksopen=true,bookmarksopenlevel=2}

\usepackage{comment}

\def\lam{\lambda}

\def\R{\mathbb R}

\def\N{\mathbb N}

\def\C{\mathbb C}
\def\HH{\mathcal H}

\def\AA{\mathcal A}

\def\la {{\lambda}}

\newcommand {\nc}   {\newcommand}
\nc {\be}   {\begin{equation}} \nc {\ee}   {\end{equation}} \nc
{\beq}  {\begin{eqnarray}} \nc {\eeq}  {\end{eqnarray}} \nc {\beqs}
{\begin{eqnarray*}} \nc {\eeqs} {\end{eqnarray*}}
\def\edc{\end{document}}

\providecommand{\abs}[1]{\lvert#1\rvert}
   
\usepackage{tikz}
\usetikzlibrary{decorations.pathmorphing,patterns,scopes,intersections,calc}
\usepackage{caption}
\usepackage{float}

\newcounter{dummy} 
\numberwithin{dummy}{section}
\newtheorem{Theorem}[dummy]{Theorem}

\newtheorem{Lemma}[dummy]{Lemma}

\newtheorem{Proposition}[dummy]{Proposition}

\numberwithin{equation}{section}

\DeclareUnicodeCharacter{2212}{-}

\begin{document}
\title[\fontsize{7}{9}\selectfont  ]{A N-dimensional  elastic$\backslash$viscoelastic transmission problem with Kelvin-Voigt damping and non smooth coefficient at the interface }
\author{Mohammad Akil$^{1}$}
\author{Ibtissam Issa$^{2,3}$}
\author{Ali Wehbe$^{2}$}
\address{$^1$ Universit\'e Savoie Mont Blanc - Chamb\'ery - France, Laboratoire LAMA}
\address{$^2$Lebanese University, Faculty of sciences 1, Khawarizmi Laboratory of  Mathematics and Applications-KALMA, Hadath-Beirut, Lebanon.}
\address{$^3$ Universit\'e Aix-Marseilles - Marseille  - France, Laboratoire I2M}
\email{mohammad.akil@univ-smb.fr}
\email{ali.wehbe@ul.edu.lb}
\email{ibtissam.issa@etu.univ-amu.fr}

\begin{abstract}
We investigate the stabilization of a multidimensional system of coupled wave equations with only one Kelvin-Voigt damping. 
Using a unique continuation result based on a Carleman estimate and a general criteria of Arendt–Batty, we prove the strong stability of the system in the absence of the compactness of the resolvent without any  geometric condition. Then, using a spectral analysis, we prove the non uniform stability of the system. Further, using frequency domain approach combined with a multiplier technique, we establish some polynomial stability results by considering different geometric conditions on the coupling and damping domains. In addition, we establish two polynomial energy decay rates of the system on a square domain where the damping and the coupling are localized in a vertical strip. 
\end{abstract}

\maketitle
\pagenumbering{roman}
\maketitle
\tableofcontents
\pagenumbering{arabic}
\setcounter{page}{1}

\section{Introduction} 
Let $\Omega\subset\R^N$ be a bounded open set with Lipschitz boundary $\Gamma$. We consider the following two wave equations coupled through velocities with a viscoelastic damping:
\begin{equation}\label{k-v}
\left\{\begin{array}{lll}
u_{tt}-\divv(a\nabla u+b(x)\nabla u_t)+c(x)y_t=0&\text{in}&\Omega\times\R^+,\\[0.1in]
y_{tt}-\Delta y -c(x) u_t=0&\text{in}&\Omega\times\R^+,
\end{array}\right.
\end{equation}
with the following initial conditions:
\begin{equation}\label{initial conditions}
u(x,0)=u_0\left(x\right),\ y(x,0)=y_0\left(x\right),\ u_t(x,0)=u_1\left(x\right),\ y_t(x,0)=y_1\left(x\right)\quad x\in\Omega,
\end{equation}
and the following boundary conditions:
\begin{equation}\label{boundary conditions}
u\left(x,t\right)=y\left(x,t\right)=0 \quad \text{ on } \Gamma\times \R^+.
\end{equation}
The functions $b,c\in L^{\infty}(\Omega)$ such that  $b:\Omega\rightarrow \R^{\ast}_{+}$ is the viscoelastic damping coefficient and $c:\Omega\rightarrow \R^{\ast}$ is the coupling function. 
The constant $a$ is a strictly positive constant.\\

The stabilization of the wave equation with localized damping has received a special attention since the seventies (see \cite{Tebou,bardos,zuazua,Dafermos}).
The stabilization of a material composed of two parts: one that is elastic and the other one that is a Kelvin-Voigt type viscoelastic material was studied extensively. This type of material is encountered in real life when one uses patches to suppress vibrations, the modeling aspect of which may be found in \cite{Banks}.
This type of damping was examined in the one-dimensional setting in \cite{liu-liu-1998,Liu2002,Liu2005}. Later on, the wave equation with Kelvin-Voigt damping in the multidimensional setting was studied. Let us consider the wave equation with Kelvin-Voigt damping given in the following system
\begin{equation}
\left\{
\begin{array}{lll}
u_{tt}-\divv(a\nabla u+b(x)\nabla u_t),\quad &\text{in}&\Omega\times\R_+^{\ast}, \\ [0.1in]
u(x,t)=0,\quad &\text{on}&\Gamma\times\R_+^{\ast},  \\ [0.1in]
u(\cdot,0)=u_0, u_t(\cdot,0)=u_1, \quad &\text{in}&\Omega
\end{array}\right.
\end{equation}
In \cite{Huang-falun}, the author proved that when the Kelvin-Voigt damping div$(b(x)\nabla u_{t})$ is globally distributed, i.e. $b(x)\geq b_0>0$ for almost all $x\in\Omega$, the wave equation generates an analytic semi-group. In \cite{Liu-Rao-2006}, the authors considered the wave equation with local visco-elastic damping distributed around the boundary of $\Omega$. They proved that the energy of the system decays exponentially to zero as t goes to infinity for all usual initial data under the assumption that the damping coefficient satisfies: $b\in C^{1,1}(\Omega)$, $\Delta b\in L^{\infty}(\Omega)$ and $|\nabla b(x)|^2\leq M_0 b(x)$ for almost every $x$ in $\Omega$ where $M_0$ is a positive constant. On the other hand, in \cite{Tebou}, the author studied the stabilization of the wave equation with Kelvin-Voigt damping. He established a polynomial energy decay rate of type $t^{-1}$ provided that the damping region is
localized in a neighborhood of a part of the boundary and verifies certain geometric condition. Also, in \cite{Nicaise-Pignotti2016}, under the same assumptions on $b$, the authors established the exponential stability of the wave equation with local Kelvin-Voigt damping localized around a part of the boundary and an extra boundary with time delay where they added an appropriate geometric condition. Later on, in \cite{Ammari_2019}, the wave equation with Kelvin-Voigt damping localized in a subdomain $\omega$ far away from the boundary without any geometric conditions was considered. The authors established a logarithmic energy decay rate for smooth initial data. In \cite{Cavalcanti}, the authors proved an exponential decay of the energy of a wave equation with two types of locally distributed mechanisms; a frictional damping and a Kelvin–Voigt damping where the location of each damping is such that none of them is able to exponentially stabilize the system. Under an appropriate geometric condition,  piecewise multiplier geometric condition in short PMGC introduced by by K. Liu in \cite{K.liu}, on a subset $\omega$ of $\Omega$ where the dissipation is effective, they proved that the energy of the system decays polynomially of type $t^{-1}$ in the absence of regularity of the Kelvin–Voigt damping coefficient $b$.  In \cite{akil-mcrf}, the authors considered a multidimensional wave equation with boundary fractional damping acting on a part of the boundary of the domain and they proved stability results under  geometric control condition (GCC in short, see Definition \ref{GCC}). 
In \cite{NNW}, the author established a polynomial energy decay rate 
of type $t^{-1}$ for smooth initial data under some geometric conditions. Also, they proved a general polynomial energy decay estimate on a bounded domain where the geometric conditions on the localized viscoelastic damping are violated and they applied it on a square domain where the damping is localized in a vertical strip.
Also,  in \cite{Robbiano-Zhang}, the authors analyzed the long time behavior of the wave equation with local Kelvin-Voigt damping where they showed the logarithmic decay rate for energy of the system without any geometric assumption on the subdomain on which the damping is effective. Furthermore, in \cite{Burq},  the author showed how perturbative approaches  and the black box strategy  allow to obtain decay rates for Kelvin-Voigt damped wave equations from quite standard resolvent estimates (for example Carleman estimates or geometric control estimates). Recently,  in \cite{Burq-Sun1}, the authors studied the energy decay rate of the Kelvin-Voigt damped wave equation with piecewise smooth damping on the multi-dimensional domain. Under suitable geometric assumptions on the support of the damping, they obtained an optimal polynomial decay rate. In 2021,  in \cite{Burq-Sun2}, they studied the decay rates for Kelvin-Voigt damped wave equations under a geometric control condition. When the damping coefficient is sufficiently smooth they showed that exponential decay follows from geometric control conditions.\\

Over the past few years, the coupled systems received a vast attention due to their potential applications. The system of coupled wave equations with only one Kelvin-Voigt damping was considered in \cite{Oquendo2017}. The authors considered the damping and the coupling coefficients to be constants and they established a polynomial energy decay rate of type  $t^{-1/2}$ and an optimality result.   In \cite{LIU2004}, exponential stability for the wave equations with local Kelvin–Voigt damping was considered where the local viscoelastic damping distributed around the boundary of the domain. They showed that the energy of the system goes uniformly and exponentially to zero for all initial data of finite energy. In \cite{Zhang2018}, the author considered the wave equation with Kelvin-Voigt damping in a non empty bounded convex domain $\Omega$ with partition $\Omega=\Omega_1\cap\Omega_2$ where the viscoelastic damping is localized in $\Omega_1$, the coupling is through a common interface. Under the condition that the damping coefficient $b$ is non smooth, she established a polynomial energy decay rate of type $t^{-1}$ for smooth initial data.  Also, in \cite{chiraz}, the authors studied the stability of coupled wave equations under Geometric Control Condition (GCC in short) where they considered one viscous damping. Finally, in \cite{Hayek2020}, the authors considered a system of weakly coupled wave equations with one or two locally internal Kelvin–Voigt damping and non-smooth coefficient at the interface.  They established some polynomial energy decay estimates under some geometric condition. The stability of wave equations coupled through velocity and with non-smooth coupling and damping coefficients is not considered yet. Also, the study of the coupled wave equations under several geometric condition is not covered. In this work, we consider the coupled system represented in \eqref{k-v}-\eqref{boundary conditions} by considering several geometric conditions (H1), (H2),  (H3), (H4),  and (H5) ( see Section \ref{Section-Poly}) where the coupling is made via velocities and with non smooth coupling and damping coefficients.  In addition, this work is a generalization of the work in \cite{akil2020stability} where the system is described by
\begin{equation}\label{eq1}
\left\{
\begin{array}{lll}
u_{tt}-\left(au_x+b(x)u_{tx}\right)_x+c(x)\ y_t&=&0,\quad (x,t)\in (0,L)\times \R^+,\\
y_{tt}-y_{xx}-c(x)\ u_t&=&0,\quad (x,t)\in (0,L)\times \R^+,
\end{array}\right.
\end{equation}
with fully Dirichlet boundary conditions and with the following initial data
\begin{equation}\label{eq3}
u(0,t)=u(L,t)=y(0,t)=y(L,t)=0,\quad \forall\ t\in \R^+,
\end{equation}
\begin{equation}\label{eq4}
u(\cdot,0)=u_0(\cdot),\ u_t(\cdot,0)=u_1(\cdot),\ y(\cdot,0)=y_0(\cdot)\quad \text{and}\quad y_t(\cdot,0)=y_1(\cdot),
\end{equation}
where
\begin{equation}\label{bc}
b(x)=\left\{\begin{array}{ccc}
b_0&\text{if}&x\in (\alpha_1,\alpha_3)\\
0&&\text{otherwise}
\end{array}
\right.
\quad \text{and}\quad c(x)=\left\{\begin{array}{ccc}
c_0&\text{if}&x\in (\alpha_2,\alpha_4)\\
0&&\text{otherwise}
\end{array}
\right.
\end{equation}
and $a>0, b_0>0$,  $c_0\in \R^{\ast}$,  and  $0<\alpha_1<\alpha_2<\alpha_3<\alpha_4<L$.   The authors considered that both the damping and the coupling coefficients are non smooth and showed that the energy of the  smooth solutions of the system decays polynomially of type $t^{-1}$. We generalize this work  to a multidimensional case and we study the stability of the system \eqref{k-v}-\eqref{boundary conditions} under several geometric control conditions.  We establish polynomial stability when there is an intersection between the damping and the coupling regions. Also, when the coupling region is a subset of the damping region and under Geometric Control Condition GCC.  Moreover, in the absence of any geometric condition, we study the stability of the system on the 2-dimensional square domain.

The paper is organized as follows: first, in Section \ref{Well-Posedness and Strong Stability}, we show that the system \eqref{k-v}-\eqref{boundary conditions} is well-posed using semi-group approach. Then, using a unique continuation result based on a Carleman estimate and a general criteria of Arendt–Batty, we prove the strong stability of the system in the absence of the compactness of the resolvent and without any geometric condition. In Section \ref{Non Uniform Stability}, using a spectral analysis, we prove the non uniform stability of the system in the case where $b(x)=b\in \R^{\ast}_ +$ and  $c(x)=c\in \R^{\ast}$.  In Section \ref{Polynomial Stability}, we establish some polynomial energy decay rates under several geometric conditions by using a frequency domain approach combined with a multiplier method. In addition, we establish two polynomial energy decay rates on a square domain where the damping and the coupling are localized in a vertical strip. 

\section{Well-Posedness and Strong Stability}\label{Well-Posedness and Strong Stability}
\subsection{Well posedness}\label{WP}
In this part, using a {\color{black}semigroup} approach, we establish the well-posedness result for the system \eqref{k-v}-\eqref{boundary conditions}.\\ 
Let $(u,u_t,y,y_t) $ be a regular solution of the system \eqref{k-v}-\eqref{boundary conditions}. The energy of the system is given by
\begin{equation}\label{KVEnergie}
E(t)=\frac{1}{2}\int_{\Omega}\left(|u_t|^2+|y_t|^2+a|\nabla u|^2+|\nabla y|^2\right)dx.
\end{equation}
A straightforward computation gives 
$$
E^{\prime}(t)=-\int_{\Omega}b(x)|\nabla u_t|^2dx\leq 0.
$$
Thus, the system \eqref{k-v}-\eqref{boundary conditions} is dissipative in the sense that its energy is a non increasing function with respect
to the time variable $t$. 
We define the energy Hilbert space $\mathcal{H}$ by 
$$
\mathcal{H}=\left(H_0^1(\Omega)\times L^2(\Omega)\right)^2
$$
equipped with the following inner product
$$
\left<U,\widetilde{U}\right>=\int_{\Omega}\left(a\nabla u\cdot\nabla\bar{\tilde{u}}+\nabla y\cdot\nabla\bar{\tilde{y}}+v\bar{\tilde{v}}+z\bar{\tilde{z}}\right)dx,
$$
for all $U=(u,v,y,z)^{\top}\in\mathcal{H}$ and $\tilde{U}=(\tilde{u},\tilde{v},\tilde{y},\tilde{z})^{\top}\in\mathcal{H}$. Finally, we define the unbounded linear operator $\mathcal{A}$ by 
\begin{equation*}
  D(\mathcal{A})=\left\{
  \begin{array}{c}
\displaystyle{U=(u,v,y,z,\omega)\in \mathcal{H}:\ v,z\in H_0^1(\Omega), \, \divv(a(x)\nabla u+b(x)\nabla v)\in L^2(\Omega)},\ y\in H^2(\Omega)\cap H_0^1(\Omega)
\end{array}
\right\}
\end{equation*}
and for all $U=(u,v,y,z,\omega)\in D(\mathcal{A})$, 
$$
\mathcal{A}(u,v,y,z)^{\top}=\begin{pmatrix}
v\\ \vspace{0.2cm} 
\displaystyle{\divv(a\nabla u+b(x)\nabla v)-c(x) z}\\ \vspace{0.2cm}
 z\\ \vspace{0.2cm}
 \Delta y+c(x) v
\end{pmatrix}.
$$
If $U=(u,u_t,y,y_t)^{\top}$ is a regular solution of system \eqref{k-v}-\eqref{boundary conditions}, then we rewrite this system as the following evolution equation
\begin{equation}\label{Cauchy}
U_t=\mathcal{A}U,\quad U(0)=U_0
\end{equation}
where $U_0=\left(u_0,u_1,y_0,y_1\right)^{\top}$.
\begin{Proposition}\label{M-Dissipative}
The unbounded linear operator $\mathcal{A}$ is m-dissipative in the energy space $\mathcal{H}$. 
\end{Proposition}
\begin{proof}
	For all $U=(u,v,y,z)^{\top}\in D(\mathcal{A})$,  we have 
	\begin{equation}\label{kvdissipation}
	\Re\left(\mathcal{A}U,U\right)_{\mathcal{H}}=-\int_{\Omega}b(x)|\nabla v|^2dx\leq 0,
	\end{equation}
	which implies that $\mathcal{A}$ is dissipative. Now, let $F=(f_1,f_2,f_3,f_4)^{\top}\in\mathcal{H}$, we prove the existence of  $$U=(u,v,y,z)^{\top}\in D(\mathcal{A})$$ unique solution of the equation
	\begin{equation}\label{kvmaximal0}
	-\mathcal{A}U=F.
	\end{equation}
	Equivalently, we have the following system
	\begin{eqnarray}
	-v&=&f_1,\label{kvmaximal1}\\
	-\divv(a\nabla u+b(x)\nabla v)+c(x)z&=&f_2,\label{kvmaximal2}\\
	-z&=&f_3,\label{kvmaximal3}\\
	-\Delta y-c(x)v&=&f_4.\label{kvmaximal4}
	\end{eqnarray}
	Inserting \eqref{kvmaximal1}, \eqref{kvmaximal3} into \eqref{kvmaximal2} and \eqref{kvmaximal4}, we get 
	\begin{eqnarray}
	-\divv(a\nabla u-b(x)\nabla f_1)&=&f_2+c(x) f_3,\label{kvmaximal5}\\
	-\Delta y&=&f_4-c(x) f_1.\label{kvmaximal6}
	\end{eqnarray}
Let $(\varphi,\psi)\in H_0^1(\Omega)\times H_0^1(\Omega)$.	Multiplying  \eqref{kvmaximal5} and \eqref{kvmaximal6} by  $\bar{\varphi}$ and $\bar{\psi}$ respectively, and integrate over $\Omega$, we obtain	
\begin{equation}\label{VP}
	a((u,v),(\varphi,\psi))=L(\varphi,\psi),\quad \forall \,(\varphi,\psi)\in H_0^1(\Omega)\times H_0^1(\Omega),
	\end{equation}
where
\begin{equation}\label{coercive}
a((u,v),(\varphi,\psi))=\int_{\Omega}\left( a\nabla u\cdot \nabla\bar{\varphi} +\nabla y\cdot\nabla\bar{\psi}\right) dx
\end{equation}
and
\begin{equation}\label{linear}
L(\varphi,\psi)=\int_{\Omega}\left(f_2+c(x) f_3\right) \bar{\varphi}dx+\int_{\Omega}b(x)\nabla f_1\cdot \nabla\bar{\varphi}dx+\int_{\Omega}\left( f_4-c(x) f_1 \right)\bar{\psi}dx.
\end{equation}
Thanks to \eqref{coercive}, \eqref{linear}, we have that $a$ is a sesquilinear,  continuous and coercive form on $(H_0^1 (\Omega) \times H_0^1 (\Omega))^2$,  and $L$ is a antilinear continuous form on $H_0^1 (\Omega) \times H_0^1 (\Omega)$. Then, using Lax-Milgram theorem, we deduce that there exists $(u,y) \in H_0^1 (\Omega) \times H_0^1 (\Omega)$ unique solution of the variational problem \eqref{VP}. By using the classical elliptic regularity, we deduce that \eqref{kvmaximal5}-\eqref{kvmaximal6} admits a unique solution $(u,y)\in H_0^1(\Omega)\times \left(H^2(\Omega)\cap H_0^1(\Omega)\right)$ such that $\divv(a\nabla u-b(x)\nabla f_1)\in L^2(\Omega)$. By taking $F=(0,0,0,0)^{\top}$ in \eqref{kvmaximal0} it is easy to see that $\ker\{\mathcal{A}\}=\{0\}$. Consequently,  we get $U=(u,-f_1,y,-f_3)^{\top}\in D(\mathcal{A})$ is a unique solution of \eqref{kvmaximal0}. Then, $\mathcal{A}$ is an isomorphism and since $\rho(\mathcal{A})$ is open set of $\mathcal{C}$ (see Theorem 6.7 (Chapter III) in \cite{Kato01}), we easily get $R(\lambda I-\mathcal{A})=\mathcal{H}$ for a sufficiently small $\lambda>0$. This, together with the dissipativeness of $\mathcal{A}$, imply that $D(\mathcal{A})$ is dense in $\mathcal{H}$ and that $\mathcal{A}$ is m-dissipative in $\mathcal{H}$ (see Theorem 4.5, 4.6 in \cite{Pazy01}). The proof is thus complete.  
\end{proof}

According to Lumer-Phillips Theorem (see \cite{Pazy01}), Proposition \ref{M-Dissipative} implies that the operator $\mathcal{A}$ generates a $C_0-$semigroup of contractions $e^{t\mathcal{A}}$ in $\mathcal{H}$ which gives the well-posedness of \eqref{Cauchy}. Then, we have the following result 
\begin{theoreme}
	For any $U_0\in\mathcal{H}$, Problem \eqref{Cauchy} admits a unique weak solution
	$$
	U(t)\in C^0(\R^+;\mathcal{H}). 
	$$
	Moreover, if  $U_0\in D(\mathcal{A}) $, then Problem \eqref{Cauchy} admits a unique strong solution $U$ satisfies
	$$
	U(t)\in C^1(\R^+,\mathcal{H})\cap C^0(\R^+,D(\mathcal{A})).
	$$
\end{theoreme}
\subsection{Strong Stability}
This subsection is devoted to study the strong stability of System \eqref{k-v}-\eqref{boundary conditions} in the sense that its energy converges to zero when $t$ goes to infinity for all initial data in $\mathcal{H}$.  The proof will be done using the unique continuation theorem based on a Carleman estimate and a general criteria of Arendt-Batty \cite{Arendt01}. For this aim, we assume that there exists  constants $b_0>0$ and $c_0>0$ and two nonempty open sets $\omega_b\subset \Omega$ and $\omega_c\subset\Omega$, such that 
\begin{equation}\label{b}
b(x)\geq b_0>0,\quad\forall\, x\in \omega_b,
\end{equation}
\begin{equation}\label{c}
c(x)\geq c_0>0,\quad\forall\, x\in \omega_c.
\end{equation}
In this part, we prove that the energy of the System \eqref{k-v}-\eqref{boundary conditions} decays to zero as $t$ tends to infinity if one of the following assumptions hold:\\

(A1) Assume that $\omega_b$ and $\omega_c$ are non-empty open subsets of $\Omega$ such that $\omega_c\subset \omega_b$ and $\meas\left(\overline{\omega_b}\cap \Gamma\right)>0$ (see Figures \ref{F1}, \ref{F2}, \ref{F3}).\\

(A2) Assume that $\omega_b$ and $\omega_c$ are non-empty open subsets of $\Omega$ such that $\omega=\omega_b\cap\omega_c\neq \emptyset$. Also, assume that $\omega$ satisfies meas$(\overline{\omega}\cap\Gamma)>0$ (see Figures \ref{F4}, \ref{F5}, \ref{F6}).\\
 
(A3) Assume that $\omega_b$ and $\omega_c$ are non-empty open subset of $\Omega$ such that $\omega=\omega_b\cap\omega_c\neq\emptyset$ , $\meas\left(\overline{\omega_b}\cap \Gamma\right)>0$ and $\omega_c$ not near the boundary (see Figure \ref{F7}). \\

(A4) Assume that $\omega_b$ is non-empty open subsets of $\Omega$ and $c(x)=c_0\in \mathbb{R}$ in $\Omega$. Also, assume that $\omega_b$ is not near the boundary (see Figure \ref{F8}). \\

\begin{figure}[!h]
\begin{floatrow}
\ffigbox{\includegraphics[scale = 0.85]{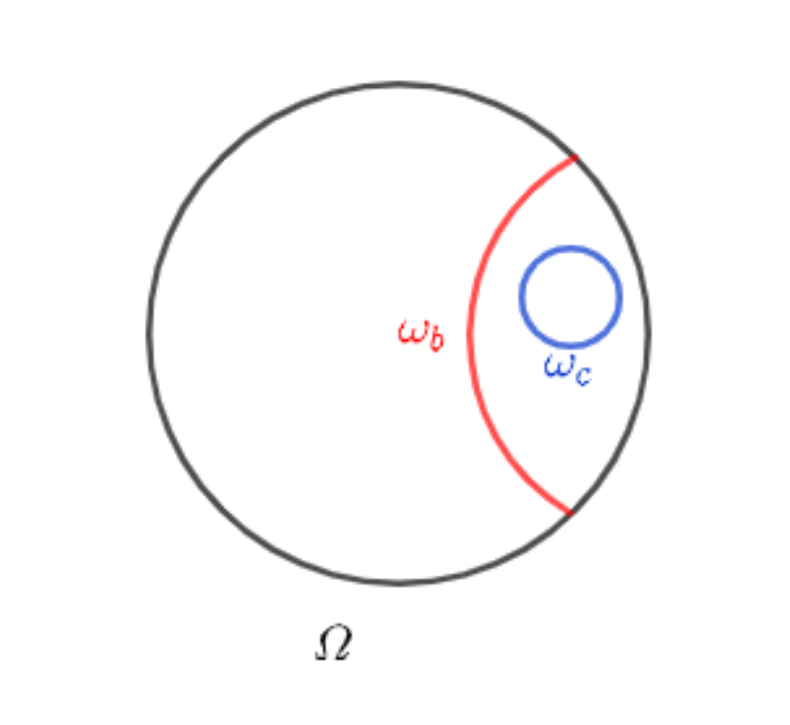}}{\caption{}\label{F1}}
\ffigbox{\includegraphics[scale = 0.85]{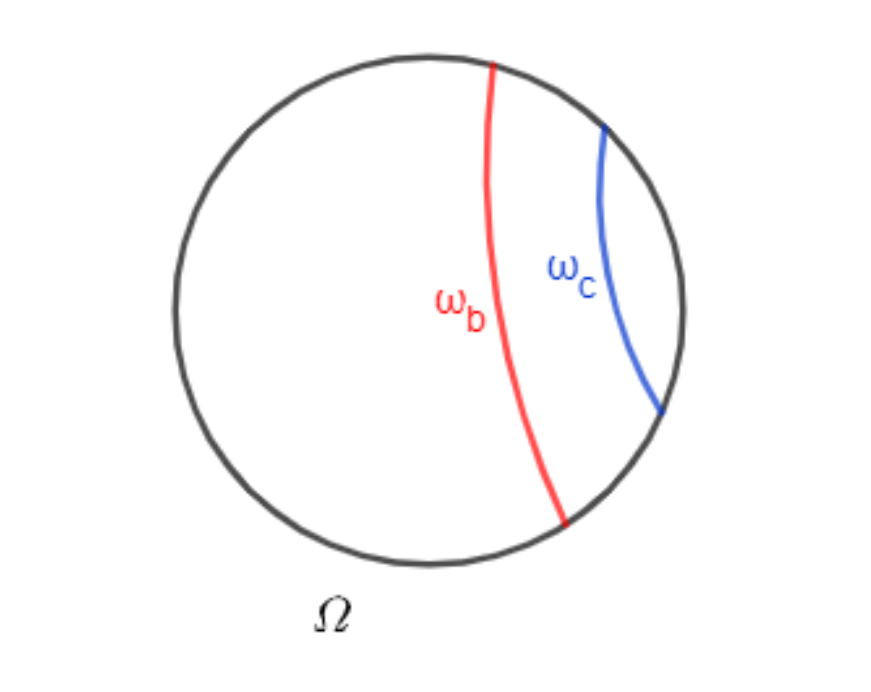}}{\caption{}\label{F2}}
\end{floatrow}
\end{figure}

\begin{figure}[!h]
\begin{floatrow}
\ffigbox{\includegraphics[scale = 0.6]{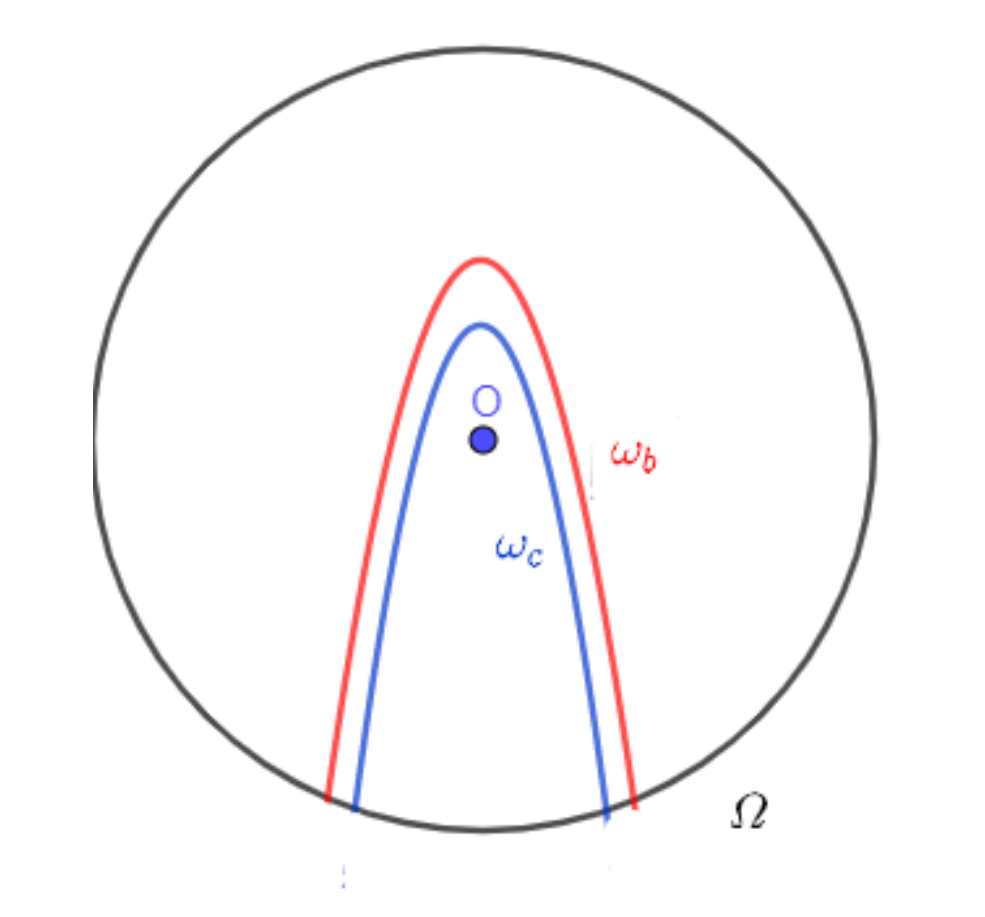}}{\caption{}\label{F3}}
\ffigbox{\includegraphics[scale = 0.8]{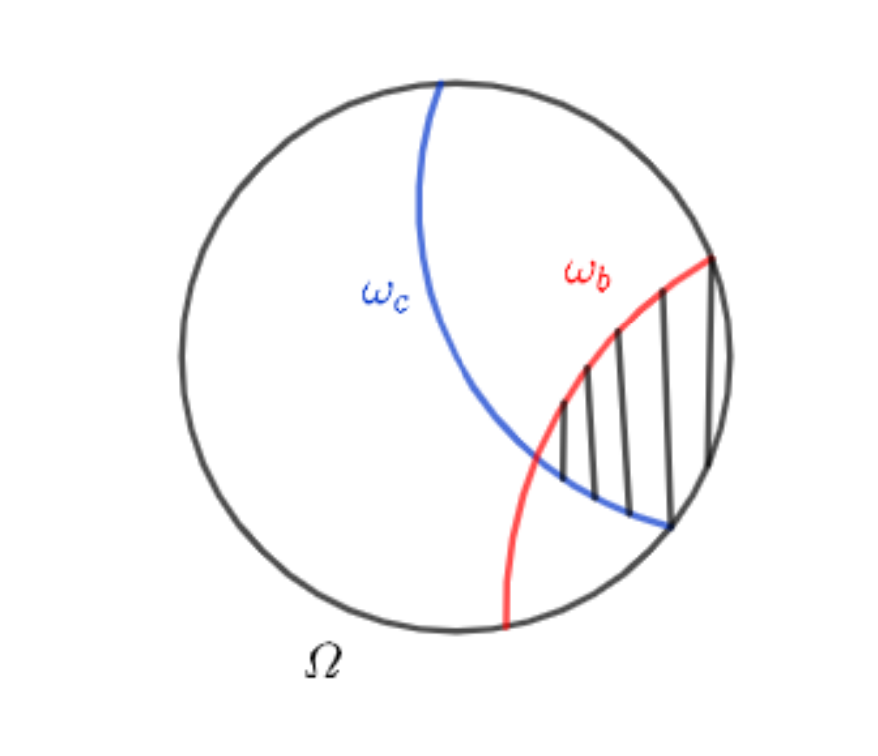}}{\caption{}\label{F4}}
\end{floatrow}
\end{figure}

\begin{figure}[!h]
\begin{floatrow}
\ffigbox{\includegraphics[scale = 0.8]{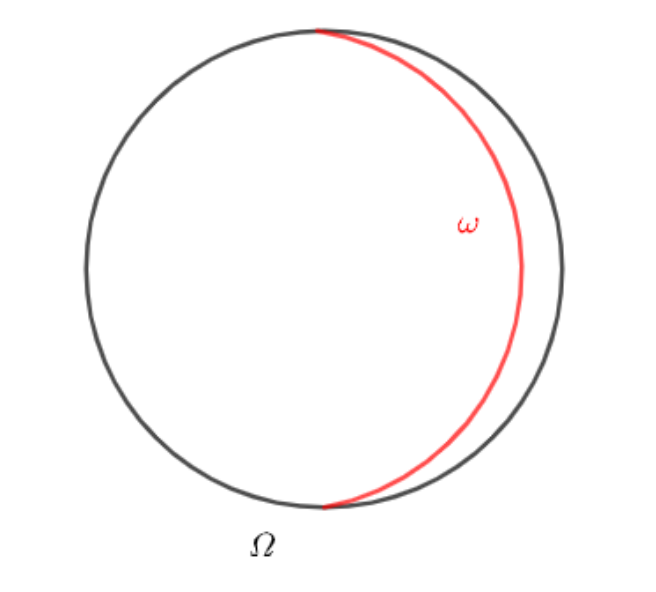}}{\caption{}\label{F5}}
\ffigbox{\includegraphics[scale = 0.55]{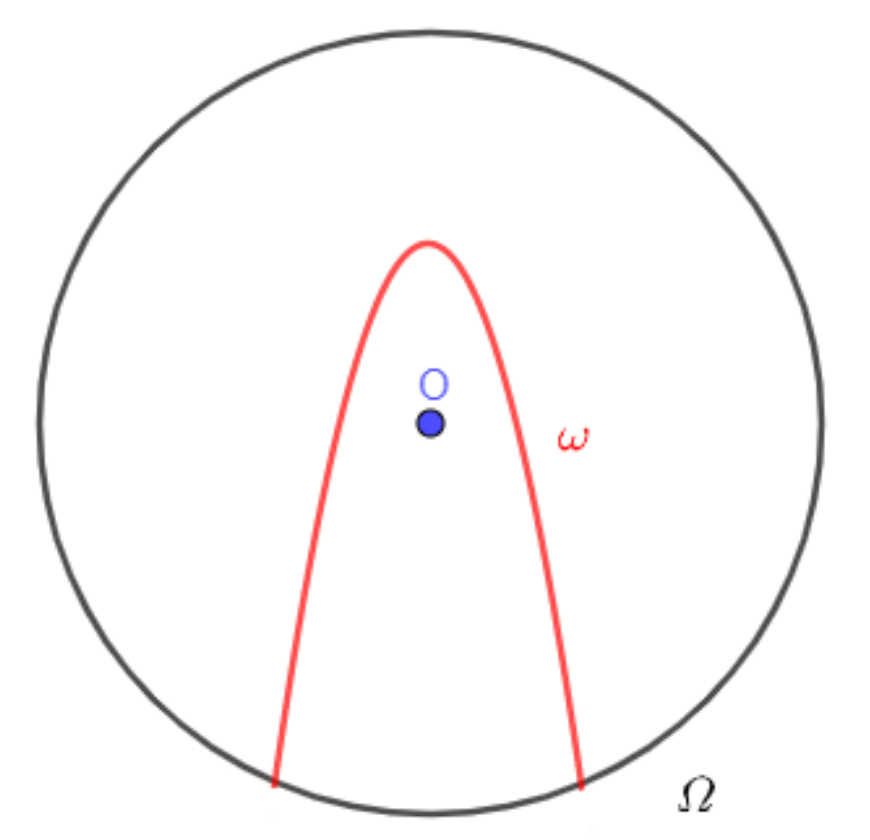}}{\caption{}\label{F6}}
\end{floatrow}
\end{figure}

\begin{figure}[!h]
\begin{floatrow}
\ffigbox{\includegraphics[scale = 0.6]{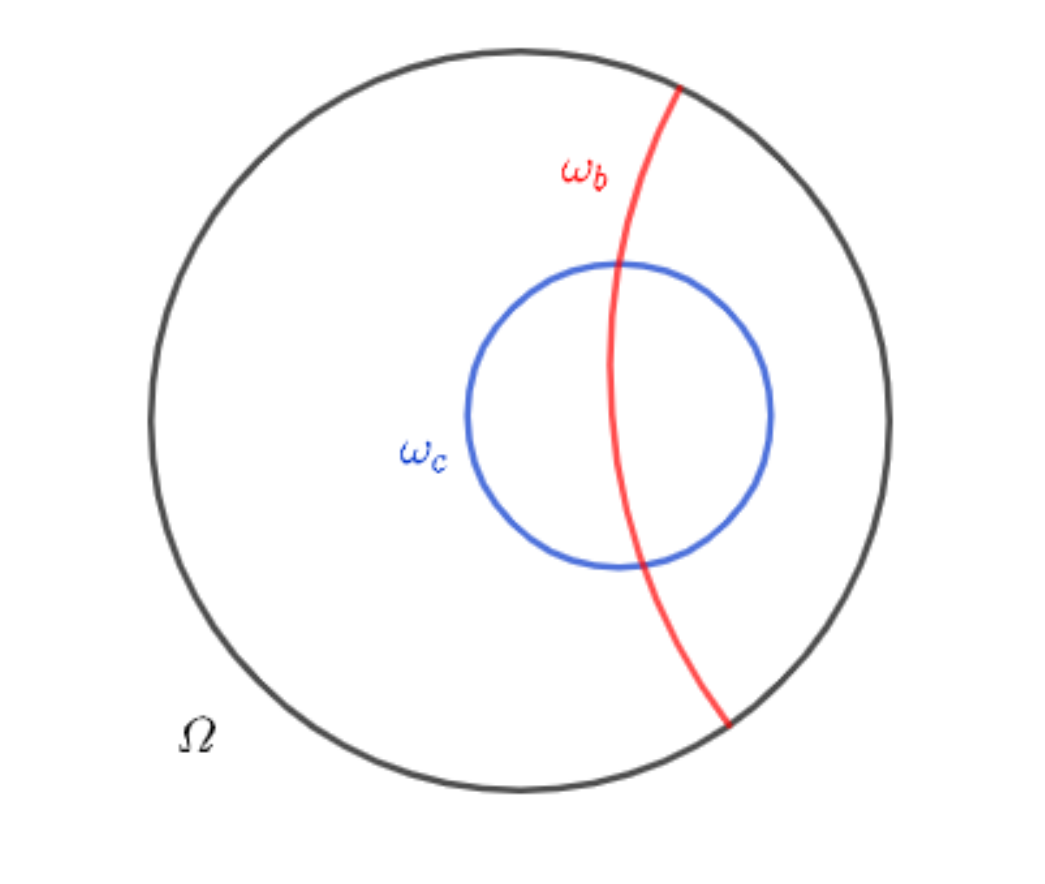}}{\caption{}\label{F7}}
\ffigbox{\includegraphics[scale = 0.85]{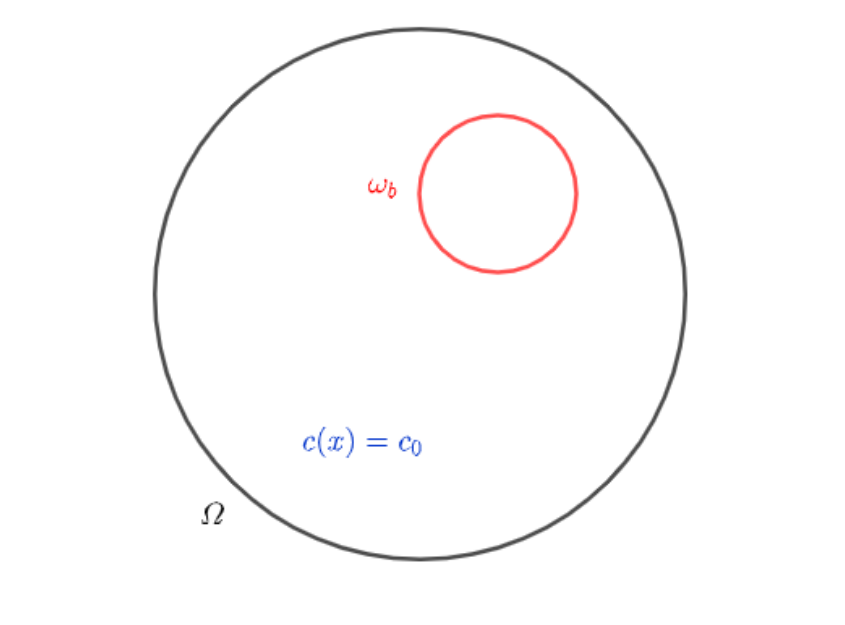}}{\caption{}\label{F8}}
\end{floatrow}
\end{figure}

\noindent We note that some of these figures were also mentioned because they are examples of the geometric conditions we consider in Section \ref{Polynomial Stability}, where we study the polynomial stability of the system.\\
Before stating the main theorem of this section, we will give the proof of a local unique continuation result for a coupled system of wave equations.

\noindent We define the following elliptic operator $P$ defined on a product space by
\begin{equation}
\begin{array}{lll}
P:H^2(V)\times H^2(V)& \rightarrow & L^2(V)\times L^2(V)\\
\hspace{1.5cm}(u,y)&\rightarrow &(\Delta u, \Delta y)
\end{array}
\end{equation}
and the following function $g$ defined by
\begin{equation}
\begin{array}{lll}
g:L^2(V)\times L^2(V)& \rightarrow &L^2(V)\times L^2(V)\\
\hspace{1.5cm}(u,y)&\rightarrow &(\frac{1}{a}(-\la^2 u+c(x)i\la y), -\la^2  y-c(x)i\la u)
\end{array}
\end{equation}
\begin{Lemma}{(See \cite{Hormander} also \cite{Hayek2020}})\label{carl}
Let $V$ be a bounded open set in $\R^N$ and let $\varphi=e^{\rho \psi}$ with $\psi\in C^{\infty}(\R^N,\R)$; $\abs{\nabla _x \psi}>0$ and $\rho>0$ large enough. Then, there exist $\tau_0$ large enough and $C>0$ such that
\begin{equation}\label{carleman-lemma}
\tau^3\|e^{\tau \varphi}u \|^2_{L^2(V)}+\tau\|e^{\tau \varphi}\nabla _x u \|^2_{L^2(V)}\leq C\|e^{\tau \varphi}\Delta u \|^2_{L^2(V)},\quad \text{for all}\,\,  u\in H_0^2(V) \,\,\text{and}\,\,  \tau>\tau_0.
\end{equation}

\end{Lemma}
\begin{Proposition}\label{carleman-prop}
Let $\Omega$ be a bounded open set in $\R^N$ and $x_0 $ be a point in $\Omega$. In a neighborhood $V$ of $x_0\in \Omega$, we take a function $f$ such that $\nabla f\neq 0$ in $\overline{V}$. Moreover, let $(u,y)\in H^2(V)\times H^2(V)$ be a solution of $P(u,y)=g(u,y)$. If $u=y=0$ in $\{x\in V; f(x)\geq f(x_0)\}$ then $u=y=0$ in a neighborhood of $x_0$.
\end{Proposition}
\begin{proof}
We call $W$ the region $\{x\in V; f(x)\geq f(x_0)\}$. We choose $V^{\prime}$ and $V^{\prime\prime}$ neighborhoods of $x_0$ such that $V^{\prime\prime}\subseteq V^{\prime}\subseteq V$, and we choose a function $\chi\in C_c^{\infty}(V^{\prime})$ such that $\chi=1$ in $V^{\prime\prime}$. Set $\tilde{u}=\chi u$ and $\tilde{y}=\chi y$. Then, $(\tilde{u},\tilde{y})\in H^2_0(V)\times H^2_0(V)$. Let $\psi=f(x)-c\abs{x-x_0}^2$ and set $\varphi=e^{\rho\psi}$. Then, apply the Carleman estimate of Lemma \ref{carl} to $\tilde{u}$ and $\tilde{y}$ respectively, then sum the two inequalities we obtain
\begin{equation}\label{carl1}
\tau^3 \int_{V\prime} e^{2\tau \varphi}\left(\abs{\tilde{u}}^2+\abs{\tilde{y}}^2\right)dx+\tau \int_{V^{\prime}}e^{2\tau \varphi}\left(\abs{\nabla\tilde{u}}^2+\abs{\nabla\tilde{y}}^2\right)dx\leq C \int_{V^{\prime}} e^{2\tau \varphi}\left(\abs{\Delta\tilde{u}}^2+\abs{\Delta\tilde{y}}^2\right)dx
\end{equation}
As $V^{\prime\prime}\subseteq V^{\prime}$ and $\chi \in C_c^{\infty}(V^{\prime})$ such that $\chi=1$ in $V^{\prime\prime}$, we get
\begin{equation}\label{carl2}
\begin{array}{l}
\displaystyle{\tau^3 \int_{V^{\prime\prime}} e^{2\tau \varphi}\left(\abs{u}^2+\abs{y}^2\right)dx+\tau \int_{V^{\prime\prime}}e^{2\tau \varphi}\left(\abs{\nabla u}^2+\abs{\nabla y}^2\right)dx\leq C \int_{V^{\prime\prime}} e^{2\tau \varphi}\left(\abs{\Delta u}^2+\abs{\Delta y}^2\right)dx}\\ [0.1in]
\displaystyle{+C \int_{V^{\prime}\backslash V^{\prime\prime}} e^{2\tau \varphi}\left(\abs{\Delta\tilde{u}}^2+\abs{\Delta\tilde{y}}^2\right)dx}.
\end{array}
\end{equation}
This implies that, 
\begin{equation}\label{carl3}
\begin{array}{l}
\displaystyle{\tau^3 \int_{V^{\prime\prime}} e^{2\tau \varphi}\left(\abs{u}^2+\abs{y}^2\right)dx\leq C \int_{V^{\prime\prime}} e^{2\tau \varphi}\left(\abs{\Delta u}^2+\abs{\Delta y}^2\right)dx
+C \int_{V^{\prime}\backslash V^{\prime\prime}} e^{2\tau \varphi}\left(\abs{\Delta\tilde{u}}^2+\abs{\Delta\tilde{y}}^2\right)dx}.
\end{array}
\end{equation}
We have that, $a\Delta u=-\la^2u+c(x)i\la y$ and $\Delta y=-\la^2u-c(x)i\la u$. Then, there exists $C_{\la,c,a}>0$ such that
\begin{equation}\label{carl4}
\left(\tau^3-C_{\la,c,a}\right) \int_{V^{\prime\prime}} e^{2\tau \varphi}\left(\abs{u}^2+\abs{y}^2\right)dx\leq C \int_{V\prime\backslash V^{\prime\prime}} e^{2\tau \varphi}\left(\abs{\Delta\tilde{u}}^2+\abs{\Delta\tilde{y}}^2\right)dx.
\end{equation}
Then, there exists $\tau>0$ large enough and $C>0$ such that
\begin{equation}\label{carl5}
\tau^3 \int_{V^{\prime\prime} }e^{2\tau \varphi}\left(\abs{u}^2+\abs{y}^2\right)dx\leq C \int_{V^{\prime}\backslash V^{\prime\prime}} e^{2\tau \varphi}\left(\abs{\Delta\tilde{u}}^2+\abs{\Delta\tilde{y}}^2\right)dx.
\end{equation}
By using that $u=y=0$ in $W$, we obtain
\begin{equation}\label{carl6}
\tau^3 \int_{V^{\prime\prime}} e^{2\tau \varphi}\left(\abs{u}^2+\abs{y}^2\right)dx\leq C \int_{S} e^{2\tau \varphi}\left(\abs{\Delta\tilde{u}}^2+\abs{\Delta\tilde{y}}^2\right)dx.
\end{equation}
where $S=V^{\prime}\backslash \left(V^{\prime\prime}\cup W\right)$.\\
For all $\varepsilon\in \R$, we set $V_{\varepsilon}=\{x\in V; \varphi(x)\leq\varphi(x_0)-\varepsilon\}$ and $V_{\varepsilon}^{\prime}=\{x\in V; \varphi(x)\geq\varphi(x_0)-\frac{\varepsilon}{2}\}$. There exists $\varepsilon$ such that $S\subset V_{\varepsilon}$. Then, choose a ball $B_0$ with center $x_0$ such that $B_0\subset V^{\prime\prime}\cap V^{\prime}_{\varepsilon}$. Hence, using \eqref{carl6}, we have 
\begin{equation}\label{carl7}
 \int_{B_0} \left(\abs{u}^2+\abs{y}^2\right)dx\leq \frac{C e^{-\tau \varphi}}{\tau^3}  \int_{S} \left(\abs{\Delta\tilde{u}}^2+\abs{\Delta\tilde{y}}^2\right)dx.
\end{equation}
Letting $\tau$ tends to infinity, we obtain $u=y=0$ in $B_0$. Hence, we reached our desired result.
\end{proof}

\begin{Theorem}{(Calderón Theorem) }\label{Calderon}
Let $\Omega$ be a connected open set in $\R^{N}$ and let $\omega\subset \Omega$, with $\omega\neq \emptyset$. If $(u,y)\in H^2(\Omega)\times H^2(\Omega)$ satisfies $P(u,y)=g(u,y)$ in $\Omega$ and $u=y=0$ in $\omega$, then $u$ and $y$ vanishes in $\Omega$.
\end{Theorem}
\begin{proof}
By setting $F=\supp u\cup \supp y$ and using Proposition \ref{carleman-prop} instead of Proposition 4.1 in the proof of the Theorem 4.2 in \cite{Rousseau-Lebeau} the result holds.
\end{proof}

\begin{Theorem}\label{kvstrongstability}
Assume that either \text{(A1)}, (A2)  (A3) or (A4) holds. Then, the $C_0-$semigroup $e^{t\mathcal{A}}$ is strongly stable in $\mathcal{H}$ in the sense that for all $U_0 \in \mathcal{H}$, the solution $U(t)=e^{t\mathcal{A}}U_0$ of  \eqref{Cauchy} satisfies
	$$
	\lim_{t\rightarrow \infty}\|e^{t\mathcal{A}}U_0\|_{\mathcal{H}}=0 .
	$$
\end{Theorem}
\noindent For the proof of Theorem \ref{kvstrongstability}, the resolvent of $\mathcal{A}$ is not compact. Then,  in order to prove this Theorem we will use a general criteria Arendt-Batty. We need to prove that the operator $\mathcal{A}$ has no pure imaginary eigenvalues and $\sigma\left(\mathcal{A}\right)\cap i\mathbb{R}$ contains only a countable number of continuous spectrum of $\mathcal{A}$. The argument for Theorem \ref{kvstrongstability} relies on the subsequent lemmas.  
\begin{Lemma}\label{ker1}
Assume that (A1) holds. Then, we have
$$
\ker\left(i\la I-\mathcal{A}\right)=\{0\},\quad \forall \la\in \R.
$$
\end{Lemma}
\begin{proof}
	From Proposition  \ref{M-Dissipative}, $0\in\rho(\mathcal{A})$. We still need to show the result for $\lambda\in\R^*$. Suppose that there exists a real number $\lambda\neq0$ and  $U=(u,v,y,z)^{\top}\in D(\mathcal{A})$ such that 
	\begin{equation}\label{kvSS1}
	\mathcal{A}U=i\lambda U.
	\end{equation}
From \eqref{kvdissipation} and \eqref{kvSS1}, we have 
	\begin{equation}\label{kvSS2}
	0=\Re\left(i\lambda\|U\|^2_{\mathcal{H}}\right)=\Re(\left<\mathcal{A}U,U\right>_{\mathcal{H}})=-\int_{\Omega}b(x)|\nabla v|^2dx.
	\end{equation}
Using the condition  \eqref{b} and Poincar\'e's inequality implies that 
	\begin{equation}\label{kvSS3}
	b(x)\nabla v=0\text{ in }\Omega \ \text{ and }\  v= 0\text{ in } \omega_{b}.
	\end{equation}
Detailing \eqref{kvSS1} and using \eqref{kvSS3}, we get the following system 
	\begin{eqnarray}
	v&=&i\lambda u\hspace{1cm}\text{in}\quad  \Omega,\label{kvss4}\\ \noalign{\medskip}
	a\Delta u-c(x) z&=&i\lambda v\hspace{1cm}\text{in }\quad \Omega,\label{kvss5}\\ \noalign{\medskip}
	z&=&i\lambda y\hspace{1cm} \text{in}\quad\Omega,\label{kvss6}\\ \noalign{\medskip}
	\Delta y+c(x)  v&=&i\la z\hspace{1cm}\text{in}\quad \Omega\label{kvss7}.
	\end{eqnarray}
From \eqref{kvSS3}, \eqref{kvss4} and \eqref{kvss5} with the assumption (A1), we get
\begin{equation}\label{kvSS4}
z=0\,\,\text{on}\,\,\,\omega_c.
\end{equation}
Inserting \eqref{kvss4} into \eqref{kvss5}, and using \eqref{kvSS3} we get
	\begin{equation}\left\{
\begin{array}{lll}
\la ^2u+a\Delta u=0 & \text{in}& \Omega\\
u=0 &\text{in} &\omega_b
\end{array}\right.
\end{equation}
Using the unique continuation theorem, we get
\begin{equation}
u=0\quad \text{on}\quad \Omega.
\end{equation}
Now, substituting \eqref{kvss6} in \eqref{kvss7}, using \eqref{kvSS4} and the definition of $c(x)$, we get
\begin{equation}\left\{
\begin{array}{lll}
\la ^2y +\Delta y=0 & \text{in}&\Omega,\\
y=0 &\text{in} &\omega_c.
\end{array}\right.
\end{equation}
Using the unique continuation theorem, we get
\begin{equation}
y=0\quad \text{on}\quad \Omega.
\end{equation}
Therefore, $U=0$. thus the proof is complete.
\end{proof}

\begin{Lemma}\label{ker2}
Assume that either (A2), (A3) or (A4) holds. Then, we have
$$
\ker\left(i\la I-\mathcal{A}\right)=\{0\},\quad \forall \la\in \R.
$$
\end{Lemma}
 \begin{proof}
	From Proposition  \ref{M-Dissipative}, $0\in\rho(\mathcal{A})$. We still need to show the result for $\lambda\in\R^*$. Suppose that there exists a real number $\lambda\neq0$ and  $U=(u,v,y,z)^{\top}\in D(\mathcal{A})$ such that 
	\begin{equation}\label{kvSS1-ker2}
	\mathcal{A}U=i\lambda U.
	\end{equation}
	From \eqref{kvdissipation} and \eqref{kvSS1-ker2}, we have 
	\begin{equation}\label{kvSS2-ker2}
	0=\Re\left(i\lambda\|U\|^2_{\mathcal{H}}\right)=\Re(\left<\mathcal{A}U,U\right>_{\mathcal{H}})=-\int_{\Omega}b(x)|\nabla v|^2dx.
	\end{equation}
Condition  \eqref{b} implies that 
	\begin{equation}\label{kvSS3-ker2}
	\nabla v=0\quad\text{in}\quad \omega_{b}.
	\end{equation}
Detailing \eqref{kvSS1-ker2} and using \eqref{kvSS3-ker2}, we get the following system 
	\begin{eqnarray}
	v&=&i\lambda u\hspace{1cm}\text{in}\quad  \Omega,\label{kvss4-ker2}\\ \noalign{\medskip}
	a\Delta u-c(x) z&=&i\lambda v\hspace{1cm}\text{in }\quad \Omega,\label{kvss5-ker2}\\ \noalign{\medskip}
	z&=&i\lambda y\hspace{1cm} \text{in}\quad\Omega,\label{kvss6-ker2}\\ \noalign{\medskip}
	\Delta y+c(x)  v&=&i\la z\hspace{1cm}\text{in}\quad \Omega\label{kvss7-ker2}.
	\end{eqnarray}
	Now, we will distinguish between the following three cases:\\
	\textbf{Case 1.} If (A2) holds. Then, by using Poincar\'e's inequality we get
	\begin{equation}\label{KKV}
v=0\,\,\text{on}\,\,\,\omega.
\end{equation}
From \eqref{kvss5-ker2}, \eqref{KKV} and using \eqref{kvss6-ker2} we get
	\begin{equation}\label{kvSS5-ker2}
y=0\,\,\text{on}\,\,\,\omega.
\end{equation}
Inserting \eqref{kvss4-ker2} and \eqref{kvss6-ker2} into \eqref{kvss5-ker2} and \eqref{kvss7-ker2} respectively, we get
	\begin{equation}\left\{
\begin{array}{lll}
\la ^2 u+a\Delta u-i\la c(x)y=0 &\text{in}& \Omega,\\[0.1in]
\la ^2 y+\Delta y+i\la c(x)u=0 & \text{in}& \Omega,\\[0.1in]
u=y=0 &\text{in} &\omega.
\end{array}\right.
\end{equation}
Then, using Theorem \ref{Calderon} we get that $u=y=0$ in $\Omega$. Thus, we deduce that $U=0$ in $\Omega$ and we reached our desired result.\\
\textbf{Case 2.} If (A3) holds. 	Then, by using Poincar\'e's inequality we get
	\begin{equation}\label{kvSS4-ker2}
v=0\,\,\text{on}\,\,\,\omega_b.
\end{equation}
Proceeding in the same way as in Case 1., we get
	\begin{equation}\left\{
\begin{array}{lll}
\la ^2 u+a\Delta u-i\la c(x)y=0 & \text{in}& \Omega,\\[0.1in]
\la ^2 y+\Delta y+i\la c(x)u=0 & \text{in}& \Omega,\\[0.1in]
u=y=0 &\text{in} &\omega.
\end{array}\right.
\end{equation}
Then, using Theorem \ref{Calderon} we get that $u=y=0$ in $\Omega$. Thus, we deduce that $U=0$ in $\Omega$ and we reached our desired result.\\
\textbf{Case 3.} Assume that ${\rm(A4)}$ holds. By differentiating \eqref{kvss5-ker2} and using the fact that $c(x)=c_0$ in $\Omega$ , we obtain 
\begin{equation*}
\partial_jy=0\quad \text{in}\quad \omega_b,\quad \forall j=1\,\cdots,N. 
\end{equation*}
Then, for all $j=1,\cdots ,N$, we have the following system 
 \begin{equation}\left\{
\begin{array}{lll}
\la ^2 \partial_ju+a\Delta \partial_ju-i\la c_0\partial_jy=0 &\text{in}& \Omega,\\[0.1in]
\la ^2 \partial_jy+\Delta \partial_jy+i\la c_0\partial_ju=0 & \text{in}& \Omega,\\[0.1in]
\partial_ju=\partial_jy=0 &\text{in} &\omega.
\end{array}\right.
\end{equation}
By applying Theorem \ref{Calderon}, we obtain 
$$
\partial_ju=\partial_jy=0\quad \text{in}\quad \Omega,\quad \forall j=1\,\cdots,N.
$$
Using the fact that $u=y=0$ on $\Gamma$, we get $u=y=0$ in $\Omega$. Consequently, $U=0$ in $\Omega$.

\end{proof}

\begin{Lemma}\label{surj}
Assume that either (A1), (A2), (A3) or (A4) holds. Then, we have
\begin{equation*}
R(i\la I-\AA)=\HH,\quad \text{for all}\quad \la\in\R.
\end{equation*}
\end{Lemma}
\begin{proof}
From Proposition \ref{M-Dissipative}, we have $0\in \rho(\mathcal{A})$. We still need to show the result for $\la\in \R^{\ast}$. Set
$F = (f_1, f_2, f_3, f_4)\in\HH $, we look for $U = (u, v, y, z)\in D(\AA)$ solution of
\begin{equation}\label{eq-2.30}
(i\la I-\AA)U=F.
\end{equation}
Equivalently, we have 
\begin{eqnarray}
v&=&i\la u-f_1,\label{surj1}\\
i\la v-\divv(a\nabla u+b(x)\nabla v)+c(x) z&=&f_2,\label{surj2}\\
z&=&i\la y-f_3,\label{surj3}\\
i\la z-\Delta y-c(x)v&=&f_4.\label{surj4}
\end{eqnarray}
Let $\left(\varphi,\psi\right)\in H_{0}^1(\Omega)\times H_0^1(\Omega)$, multiplying Equations \eqref{surj2} and \eqref{surj4} by $\bar{\varphi}$ and $\bar{\psi}$ respectively and integrating over $\Omega$, we obtain 
\begin{eqnarray}
\int_{\Omega}i\la v\bar{\varphi}dx+\int_{\Omega}a\nabla u\nabla\bar{\varphi}dx+\int_{\Omega}b(x)\nabla v\nabla\bar{\varphi}dx+\int_{\Omega}c(x)z\bar{\varphi}dx=\int_{\Omega}f_2\bar{\varphi}dx,\label{surj5}\\
\int_{\Omega}i\la z\bar{\psi}dx+\int_{\Omega}\nabla y\nabla\bar{\psi}dx-\int_{\Omega}c(x)v\bar{\psi}dx=\int_{\Omega}f_4\bar{\psi}dx.\label{surj6}
\end{eqnarray}
Substituting $v$ and $z$ in \eqref{surj1} and \eqref{surj3} into \eqref{surj5} and \eqref{surj6} and taking the sum, we obtain 
\begin{equation}\label{aL}
a\left((u,y),(\varphi,\psi)\right)={\rm L}(\varphi,\psi),\qquad \forall (\varphi,\psi)\in H_0^1(\Omega)\times H_0^1(\Omega),
\end{equation}
where 
\begin{equation*}
a\left((u,y),(\varphi,\psi)\right)=a_1\left((u,y),(\varphi,\psi)\right)+a_2\left((u,y),(\varphi,\psi)\right)
\end{equation*}
with 
\begin{equation*}
\left\{\begin{array}{l}
a_1\left((u,y),(\varphi,\psi)\right)=\displaystyle{\int_{\Omega}\left(a\nabla u\nabla\bar{\varphi}+\nabla y\nabla\bar{\psi}\right)dx+i\la\int_{\Omega}b(x)\nabla u \nabla \bar{\varphi}dx },\\[0.1in]
a_2\left((u,y),(\varphi,\psi)\right)=\displaystyle{-\la^2\int_{\Omega}\left(u\bar{\varphi}+y\bar{\psi}\right)dx+i\la\,\int_{\Omega} c(x)\left(y\bar{\varphi}-u\bar{\psi}\right)dx},
\end{array}
\right.
\end{equation*}
and 
\begin{equation*}
\begin{array}{l}
{\rm L}(\varphi,\psi)=\displaystyle{\int_{\Omega}\left(f_{2}+c(x)f_{3}+i\la f_{1}\right)\bar{\varphi}dx+\int_{\Omega}\left(f_{4}-c(x)f_{1}+i\la f_{3}\right)\bar{\psi} dx+\int_{\Omega}b(x)\nabla f_1\bar{\varphi}_xdx.}
\end{array}
\end{equation*}
Let $V=H_0^1(\Omega)\times H_0^1(\Omega)$ and $V'=H^{-1}(\Omega)\times H^{-1}(\Omega)$ the dual space of $V$. Let us consider the following operators,
$$
\left\{\begin{array}{llll}
{\rm A}:&V&\rightarrow& V'\\
&(u,y)&\rightarrow &{\rm A}(u,y)
\end{array}
\right.
\left\{\begin{array}{llll}
{\rm A_1}:&V&\rightarrow& V'\\
&(u,y)&\rightarrow &{\rm A_1}(u,y)
\end{array}
\right.
\left\{\begin{array}{llll}
{\rm A_2}:&V&\rightarrow& V'\\
&(u,y)&\rightarrow &{\rm A_2}(u,y)
\end{array}
\right.
$$
such that
\begin{equation}\label{AA1A2aa1a2}
\left\{\begin{array}{ll}
\displaystyle{\left({\rm A}(u,y)\right)(\varphi,\psi)=a\left((u,y),(\varphi,\psi)\right)},&\forall (\varphi,\psi)\in H_0^1(\Omega)\times H_0^1(\Omega),\\[0.1in]
\displaystyle{\left({\rm A_1}(u,y)\right)(\varphi,\psi)=a_1\left((u,y),(\varphi,\psi)\right)},&\forall (\varphi,\psi)\in H_0^1(\Omega)\times H_0^1(\Omega),\\[0.1in]
\displaystyle{\left({\rm A_2}(u,y)\right)(\varphi,\psi)=a_2\left((u,y),(\varphi,\psi)\right)},&\forall (\varphi,\psi)\in H_0^1(\Omega)\times H_0^1(\Omega).
\end{array}
\right.
\end{equation}
Our goal is to prove that ${\rm A}$ is an isomorphism operator. For this aim, we divide the proof into three steps.\\
{\textbf{Step 1.}} In this step, we prove that the operator ${\rm A_1}$ is an isomorphism operator. For this aim, following the second equation  of \eqref{AA1A2aa1a2} we can easily verify that $a_1$ is a bilinear continuous coercive form on $H_0^1(\Omega)\times H_0^1(\Omega)$. Then, by Lax-Milgram Lemma, the operator ${\rm A_1}$ is an isomorphism.\\

\noindent {\textbf{Step 2.}} In this step, we prove that the operator ${\rm A_2}$ is compact. According to the third equation of \eqref{AA1A2aa1a2}, we have 
$$
\abs{a_2\left((u,y),(\varphi,\psi)\right)}\leq C\|(u,y)\|_{L^2(\Omega)}\|(\varphi,\psi)\|_{L^2(\Omega)}.
$$ 
Finally, using the compactness embedding from $H_0^1(\Omega)$ to $L^2(\Omega)$ and the continuous embedding from $L^2(\Omega)$ into $H^{-1}(\Omega)$ we deduce that $A_2$ is compact.\\

\noindent From steps 1 and 2, we get that the operator ${\rm A=A_1+A_2}$ is a Fredholm operator of index zero. Consequently, by Fredholm alternative, to prove that operator ${\rm A}$ is an isomorphism it is enough to prove that ${\rm A}$ is injective, i.e. $\ker\left\{{\rm A}\right\}=\left\{0\right\}$.\\

\noindent {\textbf{Step 3.}} In this step, we prove that $\ker\{{\rm A}\}=\{0\}$. For this aim, let $\left(\tilde{u},\tilde{y}\right)\in \ker\{{\rm A}\}$, i.e. 
$$
a\left((\tilde{u},\tilde{y}),(\varphi,\psi)\right)=0,\quad \forall \left(\varphi,\psi\right)\in H_0^1(\Omega)\times H_0^1(\Omega).
$$
Equivalently, we have  
\begin{equation}\label{LUL1}
\begin{array}{ll}
\displaystyle-\la^2\int_{\Omega}\left(\tilde{u}\bar{\varphi}+\tilde{y}\bar{\psi}\right)dx+i\la\,  \int_{\Omega} c(x)\left(\tilde{y}\bar{\varphi}-\tilde{u}\bar{\psi}\right)dx+\int_{\Omega}\left( a\nabla\tilde{u}\nabla \bar{\varphi}
+ \nabla\tilde{y}\nabla\bar{\psi}\right)dx\\
\displaystyle\hspace{1cm}+i\la \int_{\Omega}b(x)\nabla\tilde{u}\nabla\bar{\varphi}dx=0.
\end{array}
\end{equation}

\noindent Taking $\varphi=\tilde{u}$ and $\psi=\tilde{y}$ in equation \eqref{LUL1}, we get   
\begin{equation*}
-\la^2\int_{\Omega}\abs{\tilde{u}}^2dx-\la^2\int_{\Omega}\abs{\tilde{y}}^2dx+a\int_{\Omega}\abs{\nabla \tilde{u}}^2dx+\int_{0}^{L}\abs{\nabla\tilde{y}}^2dx-2\la\Im\left(\int_{\Omega}c(x)\tilde{y}\bar{\tilde{u}}dx\right)+i\la \int_{\Omega}b(x)\abs{\nabla\tilde{u}}^2 dx=0.
\end{equation*}
Taking the imaginary part of the above equality, we get 
\begin{equation*}
\int_{\Omega}b(x)\abs{\nabla\tilde{u}}^2 dx=0,
\end{equation*}
 we get, 
\begin{equation}\label{LUL3}
\nabla\tilde{u}=0,\qquad \quad \text{in}\quad \omega_b.
\end{equation}
Then, we find that 
$$
\left\{\begin{array}{lll}
-\la^2\tilde{u}-a\Delta u+i\la c(x)\tilde{y}&=&0,\hspace{1cm}\text{in}\quad \Omega\\[0.1in]
-\la^2\tilde{y}-\Delta y-i\la c(x)\tilde{u}&=&0,\hspace{1cm}\text{in}\quad \Omega\label{eq-2.28}\\[0.1in]
\tilde{u}=\tilde{y}&=&0.\hspace{1cm}\text{in}\quad\omega\label{eq-2.29}
\end{array}
\right.
$$
Now, it is easy to see that the vector $\tilde{U}$ defined by
$\tilde{U}=\left(\tilde{u},i\la \tilde{u},\tilde{y},i\la \tilde{y}\right)$
belongs to $D(\mathcal{A})$ and we have $i\la \tilde{U}-\mathcal{A}\tilde{U}=0$.
Therefore, $\tilde{U}\in \ker\left(i\la I-\mathcal{A}\right)$, then by Lemmas \ref{ker1} and \ref{ker2}, we get $\tilde{U}=0$, this implies that $\tilde{u}=\tilde{y}=0$. Consequently, $\ker\left\{A\right\}=\left\{0\right\}$. 
Thus, from step 3 and Fredholm alternative, we get that the operator ${\rm A}$ is an isomorphism. It is easy to see that the operator ${\rm L}$ is continuous from $V$ to $L^2(\Omega)\times L^2(\Omega)$. Consequently, Equation \eqref{aL} admits a unique solution $(u,y)\in H_0^1(\Omega)\times H_0^1(\Omega)$. Thus, using $v=i\la u-f_1$, $z=i\la y-f_3$ and using the classical regularity arguments, we conclude that Equation \eqref{eq-2.30} admits a unique solution $U\in D\left(\mathcal{A}\right)$. The proof is thus complete. 
\end{proof}
\\

\noindent \textbf{Proof of Theorem \ref{kvstrongstability}.}  Using Lemma \ref{ker1} and \ref{ker2}, we have that $\mathcal{A}$ has non pure imaginary eigenvalues. According to Lemmas \ref{ker1}, \ref{ker2}, \ref{surj} and with the help of the closed graph theorem of Banach, we deduce that $\sigma(\mathcal{A})\cap i\mathbb{R}=\emptyset$. Thus, we get the conclusion by applying Arendt-Batty Theorem. The proof of the theorem is thus complete.
\section{Non Uniform Stability} \label{Non Uniform Stability} 
\noindent In this section, our aim is to prove the non-uniform stability of the system \eqref{k-v}-\eqref{boundary conditions}. \\
For this aim, assume that 
\begin{equation}\label{cndt}
b(x)=b\in \R^{\ast}_ +\quad\text{and}\quad c(x)=c\in \R^{\ast} .
\end{equation}

\noindent Our main result in this section is the following theorem.
\begin{Theorem}\label{non-uniform}
Under condition \eqref{cndt}. Then, the energy of the system \eqref{k-v}-\eqref{boundary conditions} does not decay uniformly in the energy space $\mathcal{H}$.
\end{Theorem}
\noindent For the proof of Theorem \ref{non-uniform}, we aim to study the asymptotic behavior of the eigenvalues of the operator $\mathcal{A}$ near the imaginary axis. First, we will determine the characteristic equation satisfied by the eigenvalues of $\mathcal{A}$. So, let $\lambda\in\C$ be an eigenvalue of $\mathcal{A}$ and let $U=(u,v,y,z)^{\top}\in D(\mathcal{A})$ be an associated eigenvector, i.e,
$$\mathcal{A}U=\la U,$$
Equivalently,
\begin{eqnarray}
v&=&\lambda u,\label{u1}\\ \noalign{\medskip}
\divv(a\nabla u+b\nabla v)-c z&=&\lambda v,\label{u2}\\ \noalign{\medskip}
z&=&\lambda y,\label{u3}\\ \noalign{\medskip}
\divv(\nabla y)+c v&=&\lambda z.\label{u4}
\end{eqnarray}
Inserting \eqref{u1} and \eqref{u3} into \eqref{u2} and \eqref{u4} respectively, we get
\begin{eqnarray}
\lambda^2u-(a+\lambda b)\Delta u+c\lambda y&=&0,\label{u5}\\ \noalign{\medskip}
\lambda^2y-\Delta y-c\lambda u&=&0.\label{u6}
\end{eqnarray}
From \eqref{u6}, we have
\begin{equation}\label{u7}
u=\frac{1}{\la c}\left[\Delta y-\la^2 y\right].
\end{equation}
Substitute \eqref{u7} in \eqref{u5}, we get
\begin{equation}\label{u8}
\left\{
\begin{array}{lll}
(a+\la b)\Delta^2y-\left[(1+a)\la^2+b\la^3\right]\Delta y+\la^2\left(\la^2+\alpha^2\right)y=0, \quad &\text{in}&\quad\Omega\\ \\
	y=\Delta y=0 \quad &\text{on}&\quad\Gamma.
\end{array}\right.
		\end{equation}
		Now, let $(\mu_k,\varphi_k)$ be, respectively, the sequence of the eigenvalues and the eigenvectors of the Laplace operator with fully Dirichlet boundary conditions in $\Omega$, i.e,
		\begin{equation}\label{u9}
	\left\{\begin{array}{lll}
	-\Delta\varphi_k=\mu_k^2\varphi_k&\text{in}&\Omega,\\
	\varphi_k=0&\text{on}&\Gamma.
	\end{array}\right.
	\end{equation}
Then by taking $y=\varphi_k$ in \eqref{u8}, we deduce the following characteristic equation
	\begin{equation}\label{charac}
	P(\la)=\la^4+b\mu_k^2\la^3+\left[(1+a)\mu_k^2+c^2\right]\la^2+b\mu_k^4\la+a\mu_k^4=0.
	\end{equation}
	\begin{Proposition}\label{prop-eigen}
	There exists $k_0\in \N^{\star}$ sufficiently large and two sequences $\left(\la_{1,k}\right)_{\abs{k}\geq k_0}$ and $\left(\la_{2,k}\right)_{\abs{k}\geq k_0}$ of simple roots of $P$ satisfying the following asymptotic behavior
	\begin{equation}\label{asy1}
	\lambda_{1,k}=i\mu_k-\frac{c^2}{2b\mu_k^2}+o\left(\frac{1}{\mu_k^3}\right)
	\end{equation}
and 
\begin{equation}\label{asy2}
	\lambda_{2,k}=-i\mu_k-\frac{c^2}{2b\mu_k^2}+o\left(\frac{1}{\mu_k^3}\right).
	\end{equation}
	\end{Proposition}
\begin{proof}
Set $\xi=\frac{\lambda}{\mu_k}$ and $\zeta_k=\frac{1}{\mu_k}$ in \eqref{charac}, we obtain
	\begin{equation}\label{u12}
	h(\xi)=b\xi^3+b\xi+a\zeta_k+c^2\xi^2\zeta_ k^3+\left(1+a\right)\xi^2\zeta_k+\xi^4\zeta_k=0.
	\end{equation}
Now, in order to find the eigenvalues of the operator $\mathcal{A}$ we need to give the roots of $h$. For this aim,  we will proceed in the following two steps. \\
\textbf{Step 1.}
Let $$ f(\xi)=b(\xi^3+\xi)\quad \text{and}\quad f_1(\xi)=a\zeta_k+c^2\xi^2\zeta_ k^3+\left(1+a\right)\xi^2\zeta_k+\xi^4\zeta_k.$$
We look for $r_k$ sufficiently small such that 
$$\left|f\right|>\left|h-f\right|=\left|f_1\right|\quad\text{on}\quad \partial D,$$
where $D=\left\{\xi\in \mathcal{C}; \left|\xi-i\right|\leq r_k\right\}$.\\
Let $\xi\in \partial D(i,r_k)$, then $\xi=i+r_k e^{i\theta}$ with $0\leq \theta \leq 2\pi$. We have 
$$f(\xi)=b(\xi^3+\xi)=b\xi r_k\left(2ie^{i\theta}+r_k e^{2i\theta}\right).$$
But, if $r_k\leq \frac{1}{2}$ then
$$\abs{\xi}\geq \abs{1-r_k}\geq\frac{1}{2}$$
and 
$$\left|2ie^{i\theta}+r_ke^{2i\theta}\right|\geq \left|2ie^{i\theta}\right|-r_k\geq \frac{3}{2}.$$
This implies that
$$ \abs{f}=\left|b(\xi^3+\xi)\right|\geq \frac{3 b r_k}{4},\quad \text{if} \quad r_k\leq\frac{1}{2}.$$
On the other hand, since $\xi$ is bounded in $D$ and $\xi_k\rightarrow 0$ we have,
$$\abs{f_1(\xi)}\leq c \,\zeta_k, \quad \text{for some constant }\quad c>0.$$
So, it is enough to choose $r_k=\frac{4c}{3b}\zeta_k$.\\
Similarly, we can find $r_k$ sufficiently small such that
$$\abs{f}>\abs{h-f}=\abs{f_1}\quad \text{on}\quad \partial D^{\prime}=\partial\left\{\xi\in \mathbb{C}; \abs{\xi+i}\leq r_k\right\}.$$
\textbf{Step 2.}
By using Step 1. and Rouch\'e's Theorem, there exists $k_0$ large enough such that for all $\abs{k}\geq k_0$ the roots of the polynoimal $h$ are close to the roots of the polynomial $f(\xi)=b(\xi^3+\xi)$.
Then, 
\begin{equation}\label{eps}
\xi_k^+=i+\varepsilon_k^+\quad\text{and}\quad \xi_k^-=-i+\varepsilon_k^-,\quad \text{with}\quad \lim_{\abs{k}\to\infty}\varepsilon^{\pm}_{k}=0.
\end{equation}
Inserting Equation \eqref{eps} in Equation \eqref{u12} and using the fact that $\la_k^{\pm}=\mu_k\xi_k^{\pm}$, we get 
	\begin{equation}\label{u13}
	\varepsilon_k^{\pm}=o\left(\frac{1}{\mu_k}\right)\quad\text{and}\quad \la_k^{\pm}=\pm i\mu_k+\tilde{\varepsilon}_k,\quad \text{where}\quad \lim_{|k|\rightarrow+\infty}\tilde{\varepsilon}_k=0.
	\end{equation}
	Multiplying Equation \eqref{charac} by $\frac{1}{\mu_k^4}$, we get 
	\begin{equation}\label{u14}
	\frac{1}{\mu_k^4}\la^4+\frac{b}{\mu_k^2}\la^3+\frac{(1+a)}{\mu_k^2}\la^2+\frac{c^2}{\mu_k^4}\la^2+b\la+a=0.
	\end{equation}
	Inserting Equation \eqref{u13} in Equation \eqref{u14}, we get 
	\begin{equation}\label{nsu17}
	\tilde{\varepsilon}_k=-\frac{c^2}{2b\mu_k^2}+o\left(\frac{1}{\mu_k^3}\right).
	\end{equation}
	The proof is thus complete.
\end{proof}

\noindent{\bf Proof of Theorem \ref{non-uniform}.} From Proposition \ref{prop-eigen} the large eigenvalues in \eqref{asy1}-\eqref{asy2} approach the imaginary axis and therefore the system \eqref{k-v}-\eqref{boundary conditions} is not uniformly stable in the energy space $\mathcal{H}$.

\section{Polynomial Stability}\label{Polynomial Stability}\label{Section-Poly}
\noindent In this section, we will study the polynomial energy decay rate of the system \eqref{k-v}-\eqref{boundary conditions}.  First, we present the definition of some geometric conditions that we encounter in this work.
\begin{definition}\label{GCC}
 For a subset $\omega$ of $\Omega$ and $T > 0$, we shall say that $(\omega,T )$ satisfies the \textbf{Geometric Control Condition} if there exists $T>0$ such that every geodesic traveling at speed one in $\Omega$ meets $\omega$ in time $t <T$.
\end{definition}

\begin{definition}\label{Gammacondition}
Saying that $\omega$ satisfies the \textbf{$\Gamma-$condition} if it contains a neighborhood in $\Omega$ of the set
$$
\left\{x\in \Gamma;\ (x-x_0)\cdot \nu(x)>0\right\},
$$
for some $x_0\in \R^n$, where $\nu$ is the outward unit normal vector to $\Gamma=\partial \Omega$.
\end{definition}

\begin{definition}\label{PMGC}
A subset $\omega$ satisfies the \textbf{Piecewise Multiplier Geometric Condition} (PMGC in short) if there exist: 
\begin{enumerate}
\item[$\bullet$] $\Omega_j\subset \Omega$ having Lipschitz boundary $\Gamma_j$.
\item[$\bullet$] $x_j\in \R^N$,\ $j=1,\cdots M$.   
\end{enumerate}
such that 
\begin{enumerate}
\item $\Omega_j\cap \Omega_i=\emptyset$ for $j\neq i$.
\item $\omega$ contains a neighborhood in $\Omega$ of the set 
$$
\bigcup_{j=1}^{M}\gamma_j(x_j)\cup \left(\Omega\backslash \bigcup_{j=1}^M\Omega_j\right)
$$
\end{enumerate}
where $\gamma_j(x_j)=\left\{x\in \Gamma_j;\ (x-x_j)\cdot \nu_j(x)>0\right\}$ and $\nu_j$ is the outward unit normal vector to $\Gamma_j$.
\end{definition}

\noindent In order to study the energy decay rate of the system, we consider the following  geometric assumptions on $\omega_b,\ \omega_c\ \text{and}\ \omega=\omega_b\cap \omega_c$:\\[0.1in]
\noindent (H1) The open subset $\omega$ verifies the GCC (see Figure \ref{F6} and Figure \ref{SGCC1}).

\noindent (H2) Assume that meas$(\overline{\omega_c}\cap \Gamma)>0$ and meas$(\overline{\omega_b}\cap \Gamma)>0$. Also, assume that $\omega_c\subset\omega_b$ and  $\omega_c$ satisfies the GCC (see Figure \ref{F3}).\\ 
(H3) Assume that $\omega_b\subset\Omega$, $\overline{\omega_c}\subset\omega_b$ such that $\Omega$ is a non-convex open set and $\omega_c$ satisfies GCC (see Figure \ref{SGCC2}).\\
(H4) Assume that $\Omega=(0,L)\times (0,L)$, $\omega_c\subset\omega_b$ such that $\omega_b=\left\{(x,y)\in \R^2; \varepsilon_1<x<\varepsilon_4\,\text{and}\,0<y<L\right\}$, $\omega_c=\left\{(x,y)\in \R^2; \varepsilon_2<x<\varepsilon_3\,\text{and}\,0<y<L\right\}$ for  $0<\varepsilon_1<\varepsilon_2<\varepsilon_3<\varepsilon_4<L$ (see Figure \ref{SQ1}).\\
(H5) Assume that $\Omega=(0,L)\times (0,L)$, $\omega_c\subset\omega_b$  such that $\omega_b=\left\{(x,y)\in \R^2; 0<x<\varepsilon_2\,\text{and}\,0<y<L\right\}$ and $\omega_c=\left\{(x,y)\in \R^2; 0<x<\varepsilon_1\,\text{and}\,0<y<L\right\}$ for $0<\varepsilon_1<\varepsilon_2<L$ (see Figure \ref{SQ2}).

\begin{figure}[!h]
\begin{floatrow}
\ffigbox{\includegraphics[scale = 0.32]{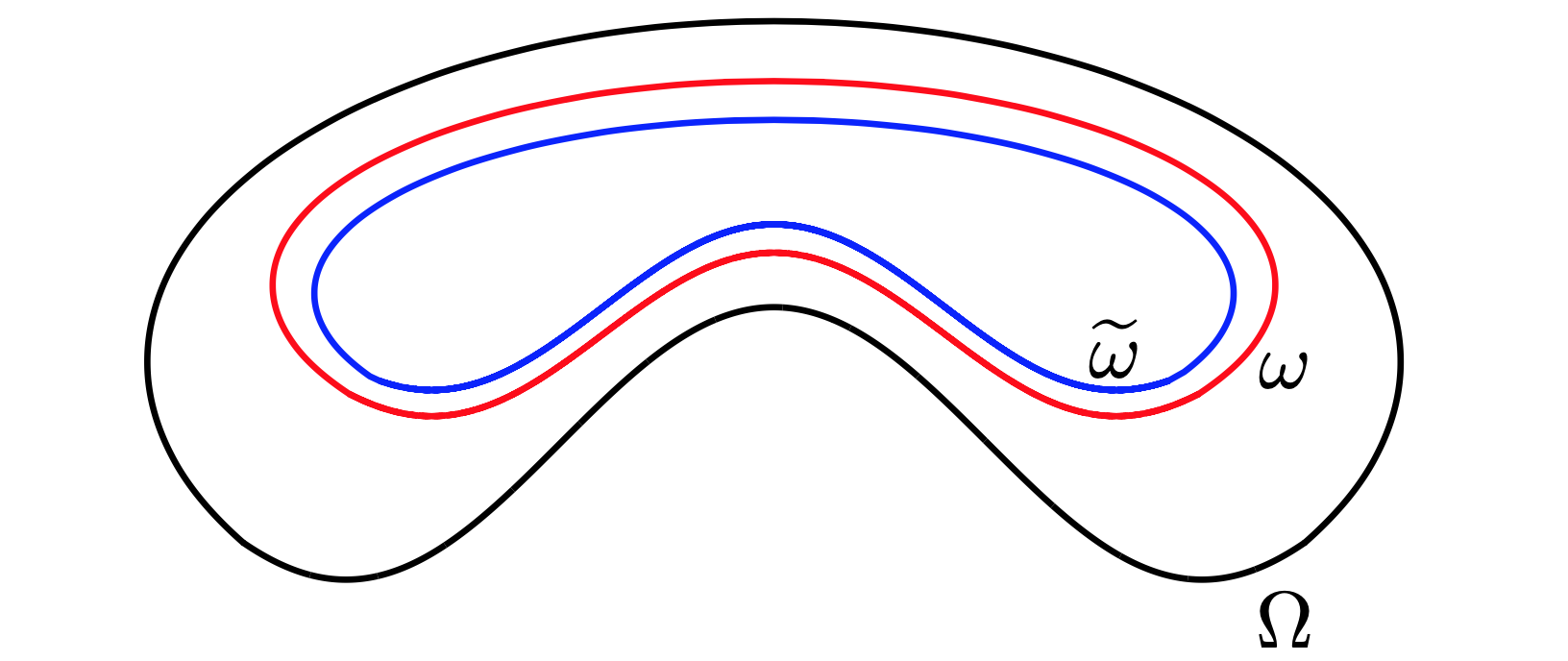}}{\caption{}\label{SGCC1}}
\ffigbox{\includegraphics[scale = 0.37]{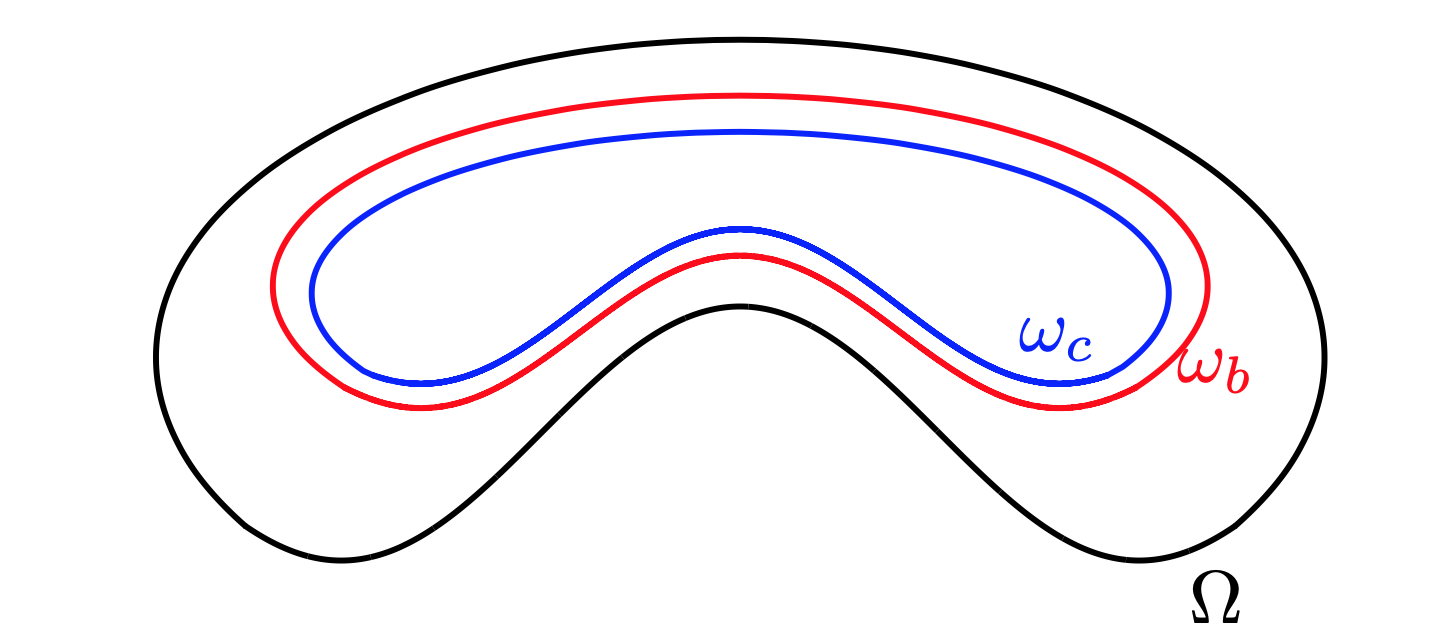}}{\caption{}\label{SGCC2}}
\end{floatrow}
\end{figure}


\begin{figure}[!h]
\begin{floatrow}
\ffigbox{\includegraphics[scale = 0.5]{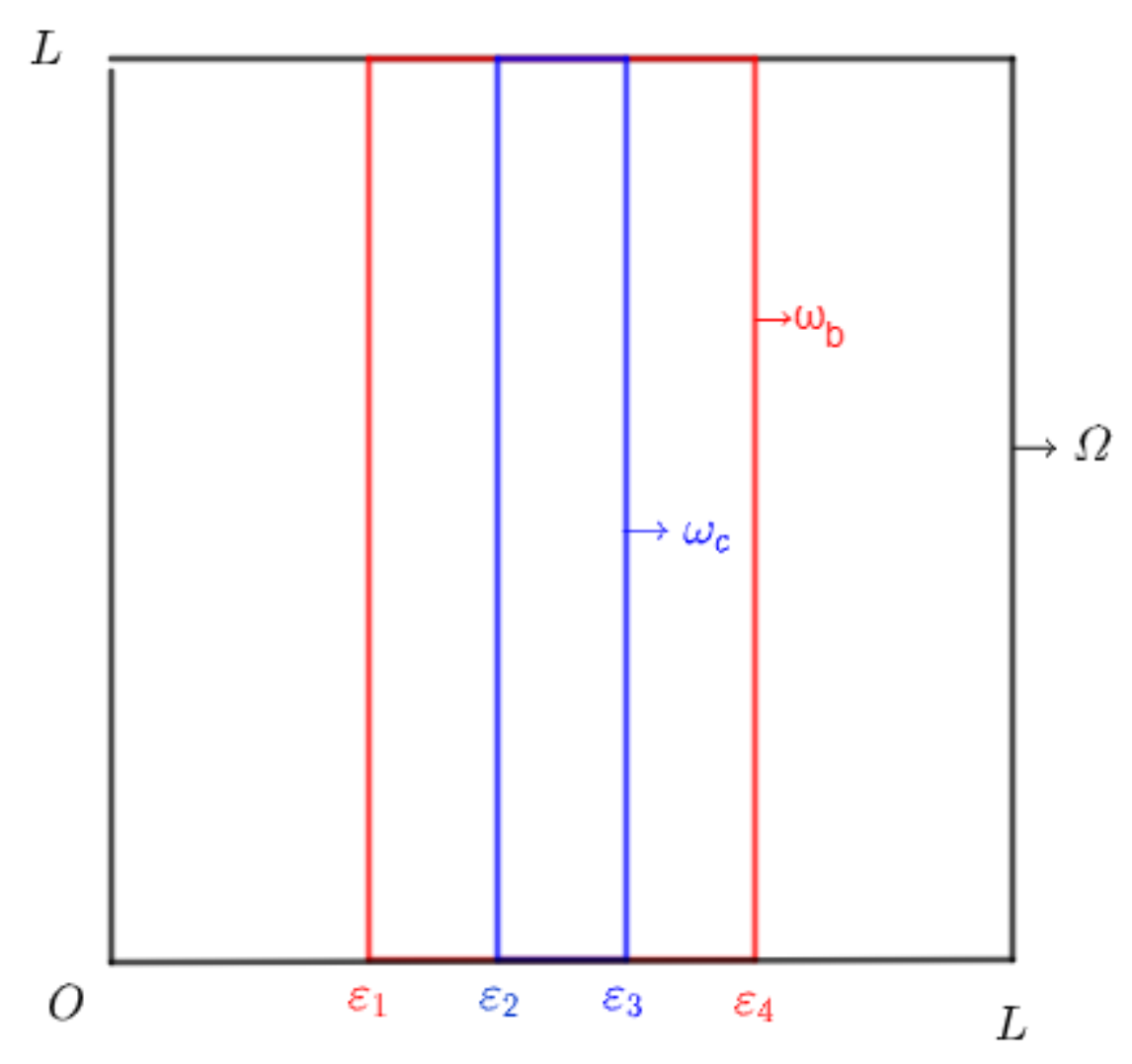}}{\caption{}\label{SQ1}}
\ffigbox{\includegraphics[scale = 0.5]{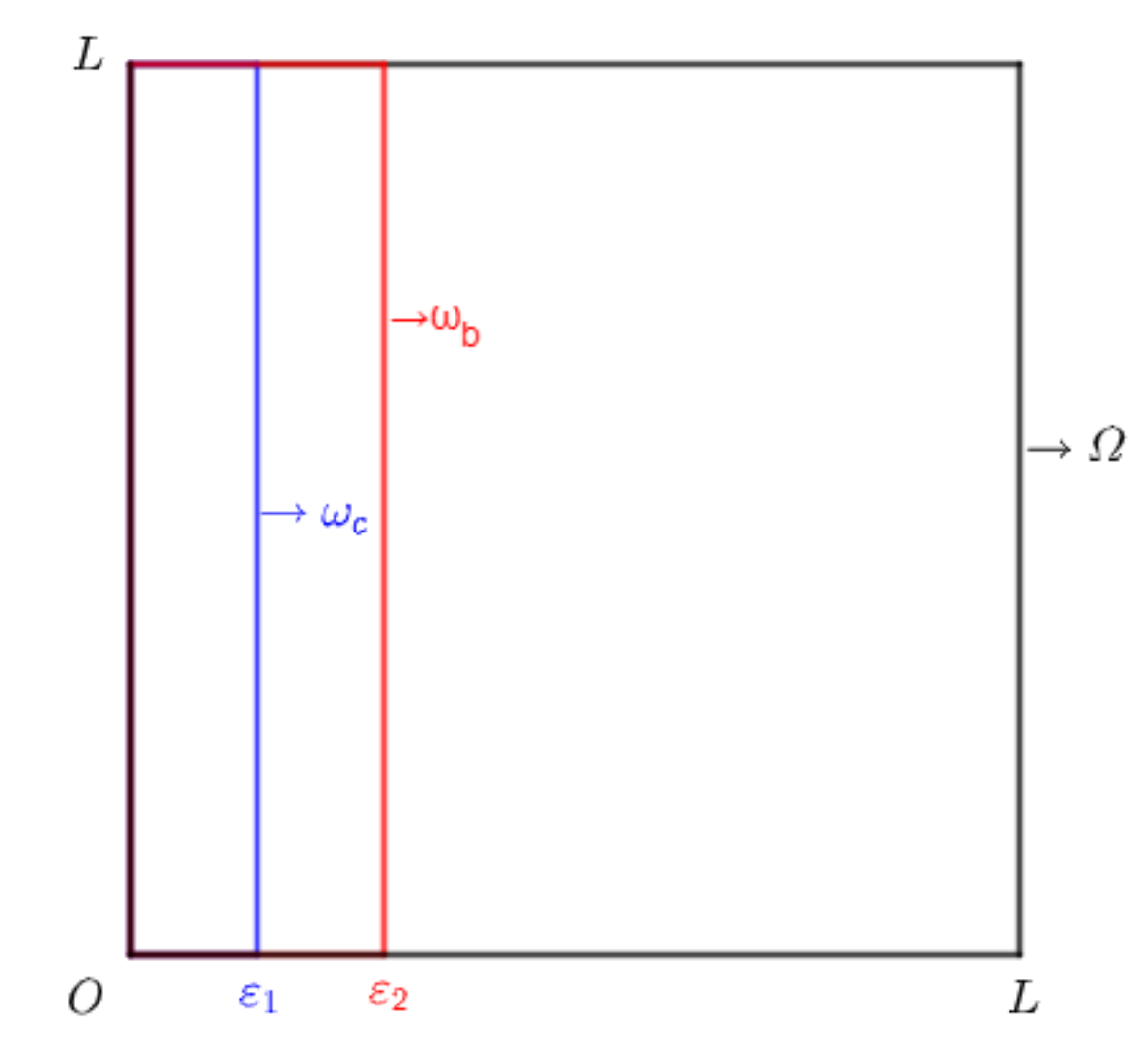}}{\caption{}\label{SQ2}}
\end{floatrow}
\end{figure}

\begin{rem}{\rm \textbf{(About Geometric Conditions and Smoothness of the boundary.)}}
\begin{enumerate}
\item If the $\Gamma-$condition applies, then it is enough to give a Lipschitz boundary conditions to $\Gamma$.
\item The GCC is an optimal condition, then we need more regularity to the $\Gamma-$condition, thus we need to take $\Gamma$ of class $C^3$. 
\item In (H1), if $\Omega$ is a convex domain.  Then,  the condition $\omega$ satisfies the GCC means that meas$\left(\overline{\omega}\cap\Gamma\right)>0$ (i.e meas$\left(\overline{\omega_b}\cap\Gamma\right)>0$ and meas$\left(\overline{\omega_c}\cap\Gamma\right)>0$).
\item In (H1), if $\Omega$ is a non convex domain.  When $\omega$ satisfies the GCC, we study the case when meas$\left(\overline{\omega_b}\cap\Gamma\right)>0$ and without the condition meas$\left(\overline{\omega_c}\cap\Gamma\right)>0$. For the case when both $\omega_b$ and $\omega_c$ are not near the boundary, we don't study this case since the strong stability remains an open problem in this case.
\item In (H4) and (H5), $\omega_b$ and $\omega_c$ does not satisfy any geometric condition. 
\end{enumerate}
\end{rem}

\noindent One of the main tools to prove the polynomial stability of \eqref{k-v}-\eqref{boundary conditions} such that the assumption (H1) holds and such that $c\in W^{1,\infty}(\Omega)$ is to use the exponential energy decay of the coupled wave equations via velocities  with two viscous dampings. We consider the following system 
\begin{equation}\label{AUX}
\left\{
\begin{array}{lll}
\varphi_{tt}-a\Delta\varphi+d(x)\varphi_t+c(x)\psi_t=0&\text{in}&\Omega\times \R^+,\\ [0.1in]
\psi_{tt}-\Delta\psi+d(x)\psi_t-c(x)\varphi_t=0&\text{in}&\Omega\times \R^+,\\ [0.1in]
\varphi(x,t)=\psi(x,t)=0&\text{on}&\Gamma\times \R^+,\\[0.1in]
(\varphi(x,0),\psi(x,0))=(\varphi_0,\psi_0)\quad \text{and}\quad (\varphi_t(x,0),\psi_t(x,0))=(\varphi_1,\psi_1)&\text{in}&\Omega.
\end{array}
\right.
\end{equation}
where $d\in W^{1,\infty}(\Omega)$ such that 
$$ d(x)\geq d_0>0 \quad\text{on}\quad\omega_\epsilon\subset\omega\subset\Omega.
$$
\noindent The energy of System  \eqref{AUX} is given by 
\begin{equation*}
E_{aux}(t)=\frac{1}{2}\left(\int_{\Omega}\abs{\varphi_t}^2+a\abs{\nabla\varphi}^2+\abs{\psi_t}^2+\abs{\nabla\psi}^2dx\right)
\end{equation*}
and by a straightforward calculation, we have 
\begin{equation*}
\frac{d}{dt}E_{aux}(t)=-\int_{\Omega}d(x)\abs{\varphi_t}^2dx-\int_{\Omega}d(x)\abs{\psi_t}^2dx\leq 0.
\end{equation*}
Thus, System \eqref{AUX} is dissipative in the sense that its energy is a non-increasing function with respect to the time variable $t$. The auxiliary energy Hilbert space of Problem \eqref{AUX} is given by 
\begin{equation*}
\mathcal{H}_{aux}=\left(H_0^1(\Omega)\times L^2(\Omega)\right)^2.
\end{equation*}
We denote by $\eta=\varphi_t$ and $\xi=\psi_t$.
The auxiliary energy space $\mathcal{H}_{aux}$ is endowed with the following norm
$$
\|\Phi\|_{\mathcal{H}_{aux}}^2=\|\eta\|^2+a\|\nabla\varphi\|^2+\|\xi\|^2+\|\nabla\psi\|^2,
$$
where $\|\cdot\|$ denotes the norm of $L^2(\Omega)$. We define the unbounded linear operator $\mathcal{A}_{aux}$ by 
\begin{equation}\label{dofao}
D(\mathcal{A}_{aux})=\left(\left(H^2(\Omega)\cap H^1_0(\Omega)\right)\times H^1_0(\Omega)\right)^2,
\end{equation}
and
$$
\mathcal{A}_{aux}(\varphi,\eta,\psi,\xi)^{\top}=\begin{pmatrix}
\eta\\ \vspace{0.2cm} 
\displaystyle{a\Delta\varphi-d(x)\eta-c(x)}\xi\\ \vspace{0.2cm}
\xi\\ \vspace{0.2cm}
 \Delta\psi-d(x)\xi+c(x)\eta
\end{pmatrix}.
$$

\noindent If $\Phi=\left(\varphi,\psi,\eta,\xi\right)$ is the state of System \eqref{AUX}, then this system is transformed into a first order evolution equation on the auxiliary Hilbert space $\mathcal{H}_{aux}$ given by
\begin{equation*}
\Phi_t=\mathcal{A}_{aux}\Phi,\quad \Phi(0)=\Phi_0,
\end{equation*}
where $\Phi_0=\left(\varphi_0,\eta_0,\psi_0,\xi_0\right)$. It is easy to see that $\mathcal{A}_{aux}$ is m-dissipative and generates a $C_0-$semigroup of contractions $\left(e^{t\mathcal{A}_{aux}}\right)_{t\geq 0}$.  
\begin{rem}\label{exponential-aux}
From \cite{chiraz}, we know that when $\omega$ satisfies the GCC condition and under the equality speed condition we have that the system of two wave equations coupled through velocity with one viscous damping is exponentially stable (see Theorem 3.1 in \cite{chiraz}). Taking this result into consideration and the fact that our system is considered with two viscous dampings and that $d,c\in W^{1,\infty}(\Omega)$,  by  proceeding with a similar proof with $a\in \R_{+}^{\ast}$ as in Theorem 3.1 in \cite{chiraz}, we can reach that the system \eqref{AUX} decays exponentially  such that there exists $M\geq 1$ and $\theta>0$ such that for all initial data $U_0\in \mathcal{H}_{aux}$, the energy of the system \eqref{AUX} satfisies the following estimation
$$E_{aux}(t)\leq Me^{-\theta t}E(0),\quad\forall\,t>0.
$$
\end{rem}

\noindent Now, we will state the main theorems in this section.
\begin{Theorem}\label{PolyStab}
Assume that the boundary $\Gamma$ is of class $C^3$. Also, assume that assumption (H1) holds and that $c\in W^{1,\infty}(\Omega)$. Then, for all initial data $U_0\in D\left(\AA\right)$, there exists a constant $C>0$ independent of $U_0$, such that the energy of the system \eqref{k-v}-\eqref{boundary conditions} satisfies the following estimation
	\begin{equation}\label{Energyp}
	E(t,U)\leq \frac{C}{t}\|U_0\|^2_{D(\mathcal{A})},\quad \forall t>0.
	\end{equation}
	
\end{Theorem}

\noindent According to Theorem \ref{bt} of  Borichev and Tomilov (see Appendix), by taking $\ell=2$, the polynomial energy decay \eqref{Energyp} holds if the following conditions 
\begin{equation}\label{C1}\tag{${\rm{C1}}$}
i\R\subset \rho(\mathcal{A}),
\end{equation}
and
\begin{equation}\label{C2}\tag{${\rm{C2}}$}
\sup_{\la\in \R}\left\|(i\la I-\AA)^{-1}\right\|_{\mathcal{L}(\mathcal{H})}=O\left(\abs{\la}^{2}\right)
\end{equation}
are satisfied. Since Condition \eqref{C1} is already proved in Lemmas \ref{ker1} and \ref{ker2}. We will prove condition \eqref{C2} by an argument of contradiction. For this purpose, suppose that \eqref{C2} is false,  then there exists \\
$\left\{\left(\la_n,U_n:=(u_n,v_n,y_n,z_n)^\top\right)\right\}\subset \R^{\ast}\times D(\AA)$ with 
\begin{equation}\label{unif}
\abs{\la_n}\to +\infty \quad \text{and}\quad \|U_n\|_{\mathcal{H}}=\|(u_n,v_n,y_n,z_n)\|_{\mathcal{H}}=1, 
\end{equation}
such that 
\begin{equation}\label{pol2-w}
\left(\la_n\right)^{2}\left(i\la_nI-\AA\right)U_n=F_n:=(f_{1,n},f_{2,n},f_{3,n},f_{4,n})^{\top}\to 0 \ \ \text{in}\ \ \mathcal{H}. 
\end{equation}
For simplicity, we drop the index $n$. Equivalently, from \eqref{pol2-w}, we have
\begin{eqnarray}
i\la u-v&=&\frac{f_1}{\lambda^{2} }\ \ \text{in}\ \ H_0^1(\Omega),\label{eq-4.4}\\ \noalign{\medskip}
i\la v-\divv(a \nabla u+b(x)\nabla v)+c(x)z&=&\frac{f_2}{\lambda^{2} }\ \ \text{in}\ \ L^2(\Omega),\label{eq-4.5}\\ \noalign{\medskip}
i\la y-z&=&\frac{f_3}{\lambda^{2} }\ \ \text{in}\ \ H_0^1(\Omega),\label{eq-4.6}\\ \noalign{\medskip}
i\la z- \divv(\nabla y)-c(x)v&=&\frac{f_4}{\lambda^{2} }\ \ \text{in}\ \ L^2(\Omega).\label{eq-4.7}
\end{eqnarray}
Here we will check the condition \eqref{C2} by finding a contradiction with \eqref{unif} such as $\left\|U\right\|_{\HH}=o(1)$. 
From Equations \eqref{unif}, \eqref{eq-4.4} and \eqref{eq-4.6} we obtain
\begin{equation}\label{unif-bdd}
\|u\|_{L^2(\Omega)}=\frac{O(1)}{\la}\quad\text{and}\quad \|y\|_{L^2(\Omega)}=\frac{O(1)}{\la}.
\end{equation}
For clarity, we will divide the proof into several lemmas.

\begin{Lemma}\label{lem1-pol}
Assume that the assumption (H1) holds. Then,  we have that the solution $(u, v, y, z) \in  D(\mathcal{A})$ of \eqref{eq-4.4}-\eqref{eq-4.7} satisfies the following estimations
\begin{equation}\label{pol1}
\|\nabla v\|_{L^2(\omega_b)}=\frac{o(1)}{\la},\quad \| v\|_{L^2(\omega_b)}=\frac{o(1)}{\la}\quad \|u\|_{L^2(\omega_b)}=\frac{o(1)}{\la^{2}}\quad\text{and}\quad \|\nabla u\|_{L^2(\omega_b)}=\frac{o(1)}{\la^{2}}.
\end{equation}
\end{Lemma}

\begin{proof}
Taking the inner product of \eqref{pol2-w} with $U$ in $\mathcal{H}$, we get
\begin{equation}\label{pol2}
\int _{\Omega} b(x)\left|\nabla v\right|^2dx=-\Re\left(\left<\AA U,U\right>_{\HH}\right)=\Re\left(\left<\left(i\la I-\AA\right)U,U\right>_{\HH}\right)=\frac{o(1)}{\la^{2}}.
\end{equation}
Then,
\begin{equation}\label{pol3}
\int _{\omega_b}\left|\nabla v\right|^2dx=\frac{o(1)}{\la^{2}}.
\end{equation}
By using Poinacr\'e inequality and Equation \eqref{pol3}, we get the second estimation in \eqref{pol1}.\\
From Equation \eqref{eq-4.4} and the second estimation in \eqref{pol1}, we obtain
$$\|u\|_{L^2(\omega_b)}=\frac{o(1)}{\la^{2}}.$$
By using \eqref{eq-4.4} and the first estimation in \eqref{pol1}, we get the last estimation. 
\end{proof}

\noindent Inserting Equations \eqref{eq-4.4} and \eqref{eq-4.6} into \eqref{eq-4.5} and \eqref{eq-4.7}, we get
\begin{eqnarray}
\lam ^{2}u+\divv(a\nabla u+b(x)\nabla v)-i\la c(x)y&=&-\frac{f_2}{\la^{2}}-c(x)\frac{f_3}{\la^{2}}-\frac{if_1}{\la},\label{pol4}\\
\la ^{2}y+\Delta y+i\la c(x)u&=&-\frac{f_4}{\la^{2}}+c(x)\frac{f_1}{\la^{2}}-\frac{if_3}{\la}.\label{pol5}
\end{eqnarray}
\begin{Lemma}\label{lem2-pol}
Assume that the assumption (H1) holds. Then, the solution $(u, v, y, z) \in  D(\mathcal{A})$ of \eqref{eq-4.4}-\eqref{eq-4.7} satisfies the following estimation
\begin{equation}\label{pol6}
\int_{\omega_{\varepsilon}}\left|\la y\right|^2dx=o(1)
\end{equation}
where $\omega_{\varepsilon}\subset\omega$  such that $\omega_{\varepsilon}$ satisfies the GCC condition.
\end{Lemma}

\begin{proof}
First, we define the function $\zeta\in C_{c}^{\infty}(\R^N)$ such that
\begin{equation}\label{zeta}
\zeta(x)=\left\{\begin{array}{ccc}
1&\text{if}&x\in \omega_{\varepsilon},\\
0&\text{if}& x\in \Omega\backslash\omega,\\
0\leq\zeta\leq 1&& elsewhere,
\end{array}\right.
\end{equation}
such that $\omega_{\varepsilon}\subset \omega$ satisfies the GCC condition. Multiply Equation \eqref{pol4} by $\la\zeta \overline{y}$ and integrate over $\Omega $ and using Green's formula, and using Equation \eqref{unif-bdd} and the fact that $\|F\|_{\mathcal{H}}=\|(f_1,f_2,f_3,f_4)\|_{\mathcal{H}}=o(1)$, we get

\begin{equation}\label{est1}
\la^3\int_{\Omega}u\zeta\overline{y}dx-\la\int_{\Omega}\left(a\nabla u+b(x)\nabla v\right)\cdot(\overline{y}\nabla\zeta +\zeta\nabla\overline{y})dx-i\int_{\Omega}c(x)\zeta\abs{\la y}^2dx=\frac{o(1)}{\la}.
\end{equation}
Estimation of the first term in \eqref{est1}. Using Cauchy-Schwarz inequality, \eqref{pol1} and \eqref{unif-bdd} we get 
\begin{equation}\label{est2}
\left|\la^3\int_{\Omega}u\zeta\overline{y}dx\right|\leq \la^3 \|u\|_{L^2(\omega)}\cdot\|y\|_{L^2(\omega)}=o(1).
\end{equation}
Estimation of the second term in \eqref{est1}. Using Cauchy-Schwarz inequality, \eqref{unif-bdd}, \eqref{pol1}, the fact that $\supp \zeta\subset \omega$, and that $\|U\|_{\mathcal{H}}=1$, we obtain the following estimations 
\begin{equation}\label{est3}
\left|\la\int_{\Omega}a\nabla u\cdot \nabla\zeta\overline{y}dx\right|\leq \la \|\nabla u\|_{L^2(\omega)}\cdot\|y\|_{L^2(\omega_{\varepsilon})}=\frac{o(1)}{\la^{2}},
\end{equation}
\begin{equation}\label{est4}
\left|\la\int_{\Omega}a\nabla u\zeta\nabla\overline{y}dx\right|\leq \la \|\nabla u\|_{L^2(\omega)}\cdot\|\nabla y\|_{L^2(\omega)}=\frac{o(1)}{\la},
\end{equation}

\begin{equation}\label{est4}
\left|\la\int_{\Omega}b(x)\nabla v\zeta\nabla\overline{y}dx\right|\leq \la \|\nabla v\|_{L^2(\omega)}\cdot\|\nabla y\|_{L^2(\omega)}=o(1),
\end{equation}

\begin{equation}\label{est5}
\left|\la\int_{\Omega}b(x)\nabla v\cdot\nabla\zeta\overline{y}dx\right|\leq \la \|\nabla v\|_{L^2(\omega_{\varepsilon})}\cdot\| y\|_{L^2(\omega_{\varepsilon})}=\frac{o(1)}{\la}.
\end{equation}
Inserting Equations \eqref{est2}-\eqref{est5} in Equation \eqref{est1}, we get that 
\begin{equation}
i\int_{\Omega}c(x)\zeta\abs{\la y}^2dx=o(1).
\end{equation}
Using the definiton of the function $c$ and $\zeta$, we obtain our desired result.
\end{proof}

\begin{lemma}
For any $\la\in \R$, the solution $\left(\varphi,\psi\right)\in ((H^2(\Omega)\cap H^1_0(\Omega))^2$ of the system 
\begin{equation}\label{aux2}
\left\{
\begin{array}{lll}
\lam ^{2}\varphi+a\Delta\varphi-i\lam d(x)\varphi-i\lam c(x)\psi=u,&\text{in} &\Omega\\ \\
\lam ^{2}\psi+\Delta\psi-i\lam d(x)\psi+i\lam c(x)\varphi=y, &\text{in}& \Omega\\ \\
\varphi=\psi=0,&\text{on}&\Gamma\\ \\
\end{array}\right.
\end{equation}
satisfies the following estimation
\begin{equation}\label{bdd}
\|\lam\varphi\|_{L^{2}(\Omega)}+\|\nabla\varphi\|_{L^{2}(\Omega)}+\|\lam\psi\|_{L^{2}(\Omega)}+\|\nabla\psi\|_{L^{2}(\Omega)}\leq M\left(\| u\|_{L^{2}(\Omega)}+\| y\|_{L^{2}(\Omega)}\right).
\end{equation}
where $\beta\geq 0$.
\end{lemma}
\begin{proof}
Using Remark \ref{exponential-aux}, then the resolvent set of the associated operator $\mathcal{A}_{aux}$ contains $i\R$ and $\left(i\lam I-\mathcal{A}_{aux}\right)^{-1}$ is uniformly bounded on the imaginary axis. Consequently, there exists $M>0$ such that
\begin{equation}\label{expo}
\sup_{\la\in \R}\|\left(i\lam I-\mathcal{A}_{aux}\right)^{-1}\|_{\mathcal{L}\left(\mathcal{H}_{aux}\right)}\leq M<+\infty.
\end{equation}
Now, since $(u,y)\in H^1_0(\Omega)\times H^1_0(\Omega)$, then $(0,-u,0,-y)$ belongs to $\HH_{aux}$, and from \eqref{expo}, there exists $(\varphi,\eta,\psi,\xi)\in D(\mathcal{A}_{aux})$ such that $\left(i\lam I-\mathcal{A}_{aux}\right)(\varphi,\eta,\psi,\xi)=(0,-u,0,-y)^{\top}$ $i.e.$

\begin{eqnarray}
i\la \varphi-\eta&=&0,\label{x1}\\ \noalign{\medskip}
i\la \eta-a\Delta\varphi+d(x)\eta+c(x)\xi&=&-u,\label{x2} \\ \noalign{\medskip}
i\la \psi-\xi&=&0,\label{x3}\\ \noalign{\medskip}
i\la \xi-\Delta \psi+d(x)\xi-c(x)\eta &=&-y,\label{x4}
\end{eqnarray}
such that
\begin{equation}\label{x4.1}
\|(\varphi,\eta,\psi,\xi)\|_{\HH_a}\leq M\left(\|u\|_{L^{2}(\Omega)}+\|y\|_{L^{2}(\Omega)}\right).
\end{equation}
From equations \eqref{x1}-\eqref{x4.1}, we deduce that $(\varphi,\psi)$ is a solution of \eqref{aux2} and we have 
\begin{equation}\label{AUxuy}
\|\lam\varphi\|_{L^{2}(\Omega)}+\|\nabla\varphi\|_{L^{2}(\Omega)}+\|\lam\psi\|_{L^{2}(\Omega)}+\|\nabla\psi\|_{L^{2}(\Omega)}\leq  M\left(\| u\|_{L^{2}(\Omega)}+\| y\|_{L^{2}(\Omega)}\right).
\end{equation}
Thus, we get our desired result.
\end{proof}

\begin{Lemma}\label{lem3-pol}
Assume that  the assumption (H1) holds. Then,  the solution $(u, v, y, z) \in  D(\mathcal{A})$ of \eqref{eq-4.4}-\eqref{eq-4.7} satisfies the following estimations
\begin{equation}
\int_{\Omega}\left|\la u\right|^2dx=o(1)\quad \text{and}\quad\int_{\Omega}\left|\la y\right|^2dx=o(1).
\end{equation}
\end{Lemma}
\begin{proof}
For clarity, we will divide the proof of this Lemma into two steps.\\
\textbf{Step 1.}
Multiply \eqref{pol4} by $\la^2\overline{\varphi}$ and integrate over $\Omega$, and using Green's formula, Equation \eqref{bdd}, and the fact that $\|F\|_{\mathcal{H}}=\|(f_1,f_2,f_3,f_4)\|_{\mathcal{H}}=o(1)$, we obtain
\begin{equation}\label{est6}
\int_{\Omega}\left(\la^2\overline{\varphi}+a\Delta\overline{ \varphi}\right)\la^2 u dx-\la^{2}\int_{\Omega}b(x)\nabla v\cdot\nabla\overline{\varphi}dx-\int_{\Omega} i\la^3 c(x)y\overline{\varphi}dx=\frac{o(1)}{\la}.
\end{equation}
From Equations \eqref{pol1} and \eqref{bdd}, we obtain 
\begin{equation}\label{est7}
\left|\la^{2}\int_{\Omega}b(x)\nabla v\cdot\nabla\overline{\varphi}dx\right|=o(1).
\end{equation}
By using Equation \eqref{est7} in \eqref{est6}, we get
\begin{equation}\label{est8}
\int_{\Omega}\left(\la^2\overline{\varphi}+a\Delta\overline{ \varphi}\right)\la^2 u dx-\int_{\Omega} i\la^3 c(x)y\overline{\varphi}dx=o(1).
\end{equation}
Now, from System \eqref{aux2}, we have that
\begin{equation}\label{est9}
\la^2\overline{\varphi}+a\Delta\overline{ \varphi}=-i\lam d(x)\overline{\varphi}-i\lam c(x)\overline{\psi}+\overline{u}.
\end{equation}
Inserting Equation \eqref{est9} into \eqref{est8}, we obtain
\begin{equation}\label{est10}
\int_{\Omega}\left|\la u\right|^2dx-i\la^3\int_{\Omega} d(x)u\overline{\varphi}dx-i\lam^3\int_{\Omega} c(x)u\overline{\psi}dx-\int_{\Omega} i\la^3 c(x)y\overline{\varphi}dx=o(1).
\end{equation}
By using \eqref{pol1} and \eqref{bdd}, we get
\begin{equation}\label{est11}
\left|i\la^3\int_{\Omega}d(x)u\overline{\varphi}dx\right|\leq \la^3 \|u\|_{L^2(\omega_{\varepsilon})}\cdot\|\varphi\|_{L^2(\Omega)}=\frac{o(1)}{\la}.
\end{equation}
Now, inserting Equation \eqref{est11} into \eqref{est10}, we get
\begin{equation}\label{est122}
\int_{\Omega}\left|\la u\right|^2dx-i\lam^3\int_{\Omega} c(x)u\overline{\psi}dx-i\la^3\int_{\Omega}  c(x)y\overline{\varphi}dx=o(1).
\end{equation}
\textbf{Step 2.}\\
Multiply Equation \eqref{pol5} by $\la^2 \overline{\psi}$, integrate over $\Omega$, using Green's formula, and the fact that $\|F\|_{\mathcal{H}}=\|(f_1,f_2,f_3,f_4)\|_{\mathcal{H}}=o(1)$, we obtain
\begin{equation}\label{est-1}
\int_{\Omega}\left(\la^2\overline{\psi}+\Delta\overline{ \psi}\right)\la^2 y dx+i\la^3\int_{\Omega}  c(x)u\overline{\psi}dx=\frac{o(1)}{\la}.
\end{equation}
From System \eqref{aux2}, we have
\begin{equation}\label{est-2}
\la^2\overline{\psi}+\Delta\overline{ \psi}=-i\lam d(x) \overline\psi+i\lam c(x) \overline\varphi+ \overline y
\end{equation}
Inserting \eqref{est-2} into \eqref{est-1}, we get
\begin{equation}\label{est-3}
\int_{\Omega}\abs{\la y}^2dx-i\lam^3\int_{\Omega}d(x) y\overline\psi dx+i\lam ^3\int_{\Omega}c(x) y\overline\varphi dx+i\la^3\int_{\Omega}  c(x)u\overline{\psi}dx=\frac{o(1)}{\la}.
\end{equation}
Using Cauchy-Schwarz inequality, Lemma \ref{lem2-pol}, and Equation \eqref{bdd}
\begin{equation}\label{est-4}
\left|i\lam^3\int_{\Omega}d(x) y\overline\psi\right|=o(1).
\end{equation}
Inserting \eqref{est-4}into \eqref{est-3}, we get
\begin{equation}\label{est-5}
\int_{\Omega}\abs{\la y}^2dx+i\lam ^3\int_{\Omega}c(x) y\overline\varphi dx+i\la^3\int_{\Omega}  c(x)u\overline{\psi}dx=o(1).
\end{equation}
Adding Equations \eqref{est122} and \eqref{est-5}, we get
\begin{equation}
\int_{\Omega}\left|\la u\right|^2dx=o(1)\quad \text{and}\quad\int_{\Omega}\left|\la y\right|^2dx=o(1).
\end{equation}
Thus, the proof of the Lemma is completed.
\end{proof}
\begin{Lemma}\label{lem4-pol}
Assume that the assumption (H1)  holds. Then,  the solution $(u, v, y, z) \in  D(\mathcal{A})$ of \eqref{eq-4.4}-\eqref{eq-4.7} satisfies the following estimations
\begin{equation}
\int_{\Omega}\left|\nabla u\right|^2dx=o(1)\quad \text{and}\quad\int_{\Omega}\left|\nabla y\right|^2dx=o(1).
\end{equation}
\end{Lemma}
\begin{proof}
Multiply Equation \eqref{pol4} by $\overline{u}$, integrating over $\Omega$, Green's formula, Equation \eqref{unif-bdd} and the fact that  $\|F\|_{\mathcal{H}}=\|(f_1,f_2,f_3,f_4)\|_{\mathcal{H}}=o(1)$, we obtain
\begin{equation}
\int_{\Omega}\left|\la u\right|^2dx-a\int_{\Omega}\left|\nabla u\right|^2dx-\int_{\Omega}b(x)\nabla v\cdot \nabla \overline{u}dx-\int_{\Omega}i\la c(x)y\overline{u}dx=\frac{o(1)}{\la^{2}}
\end{equation}
Using Equation \eqref{pol1} and Lemma \ref{lem3-pol}, we obtain
\begin{equation}
\int_{\Omega}\left|\nabla u\right|^2dx=o(1).
\end{equation}
Multiplying Equation \eqref{pol5} by $\overline{y}$ and proceeding in a similar way as above, we get
\begin{equation}
\int_{\Omega}\left|\nabla y\right|^2dx=o(1).
\end{equation}
\end{proof}

\noindent \textbf{Proof of Theorem \ref{PolyStab}.} 
Consequently, from the results of Lemmas \ref{lem3-pol} and \ref{lem4-pol}, we obtain  
\begin{equation*}
\int_{\Omega}\left(|v|^2+|z|^2+a\,|\nabla u|^2+ |\nabla y|^2\right)dx
=o\left(1\right).
\end{equation*}
Hence $\|U\|_{\HH}=o(1)$, which contradicts \eqref{unif}. Consequently, condition \eqref{C2} holds. This implies that the energy decay estimation \eqref{Energyp}. The proof is thus complete.
\\[0.1in]
\begin{rem}
In the case when $d,c\in L^{\infty}(\Omega)$ such that they are discontinuous functions, we didn't find any result on the stability of the system \eqref{AUX}. But, we can conjecture that the system \eqref{AUX} is exponentially stable. Further,  we have that the system \eqref{AUX} with $d,c$ are discontinuous functions is exponentially stable in the dimension 1 (see \cite{akil2020stability}).
\end{rem}

\noindent One of the main tools to prove the polynomial stability of the system \eqref{k-v}-\eqref{boundary conditions} when one of the assumptions (H2), (H3), (H4) or (H5) holds is to use the exponential or polynomial decay of the wave equation with  viscous damping. We consider the following system 
\begin{equation}\label{AUXXX}
\left\{\begin{array}{lll}
\varphi_{tt}-\Delta \varphi+\mathds{1}_{\omega_c}(x)\varphi_t=0&\text{in}&\Omega \times (0,+\infty)\\
\varphi=0&\text{in}&\Gamma\times (0,+\infty)\\
\varphi(\cdot,0)=\varphi_0,\quad \varphi_t(\cdot,0)=\varphi_1.
\end{array}
\right.
\end{equation}

\begin{rem}\label{AUX-1W}{\rm \textbf{(About System \eqref{AUXXX})}}
\begin{enumerate}
\item If (H2) or (H3) holds, system \eqref{AUXXX} is exponentially stable (see \cite{Burq1997ContrlabilitED} and Lemma 3.8 in \cite{NNW}). 
\item If (H4) holds, the energy of the wave equation \eqref{AUXXX} with local viscous damping decays polynomially as $t^{-1}$ for smooth initial data (see Example 3 in \cite{Liu2005}).
\item If (H5) holds,  the energy of the wave equation \eqref{AUXXX} with local viscous damping decays polynomially as $t^{-\frac{4}{3}}$ for smooth initial data (see \cite{Stahn2017}).
\end{enumerate}
\end{rem}

\begin{Theorem}\label{Theorem2}
Assume that assumption (H2) or (H3) holds. Also, assume that the energy of the system \eqref{AUXXX} is exponentially stable. Then, for all initial data $U_0\in D(\mathcal{A})$, there exists a constant $C_2>0$ independent of $U_0$,  the energy of the system \eqref{k-v}-\eqref{boundary conditions} satisfies the following estimation
\begin{equation}\label{Energyp2}
	E(t,U)\leq \frac{C_2}{t}\|U_0\|^2_{D(\mathcal{A})},\quad \forall t>0.
	\end{equation}
\end{Theorem}

\noindent Following Theorem \ref{bt} of Borichev and Tomilov (see Appendix), the polynomial energy decay \eqref{Energyp2} holds if \eqref{C1} and 
\begin{equation}\label{C3}\tag{${\rm{C3}}$}
\sup_{\la\in \R}\left\|(i\la I-\AA)^{-1}\right\|_{\mathcal{L}(\mathcal{H})}=O\left(\abs{\la}^{2}\right)
\end{equation}
holds. Since Condition \eqref{C1} is already proved. We will prove condition \eqref{C3} by an argument of contradiction. For this purpose, suppose that \eqref{C3} is false,  then there exists \\
$\left\{\left(\la_n,U_n:=(u_n,v_n,y_n,z_n)^\top\right)\right\}\subset \R^{\ast}\times D(\AA)$ with 
\begin{equation}\label{n1}
\abs{\la_n}\to +\infty \quad \text{and}\quad \|U_n\|_{\mathcal{H}}=\|(u_n,v_n,y_n,z_n)\|_{\mathcal{H}}=1, 
\end{equation}
such that 
\begin{equation}\label{n2}
\la_n^{2}\left(i\la_nI-\AA\right)U_n=F_n:=(f_{1,n},f_{2,n},f_{3,n},f_{4,n})^{\top}\to 0 \ \ \text{in}\ \ \mathcal{H}. 
\end{equation}
For simplicity, we drop the index $n$. Equivalently, from \eqref{n2}, we have
\begin{eqnarray}
i\la u-v&=&\la^{-2}f_1 \ \text{in}\ \ H_0^1(\Omega),\label{n3}\\ \noalign{\medskip}
i\la v-\divv(a \nabla u+b(x)\nabla v)+c(x)z&=&\la^{-2}f_2\ \ \text{in}\ \ L^2(\Omega),\label{n4}\\ \noalign{\medskip}
i\la y-z&=&\la^{-2}f_3\ \ \text{in}\ \ H_0^1(\Omega),\label{n5}\\ \noalign{\medskip}
i\la z- \divv(\nabla y)-c(x)v&=&\la^{-2}f_4\ \ \text{in}\ \ L^2(\Omega).\label{n6}
\end{eqnarray}
Here we will check the condition \eqref{C3} by finding a contradiction with \eqref{n1} such as $\left\|U\right\|_{\HH}=o(1)$. 
From Equations \eqref{n1}, \eqref{n3} and \eqref{n5} we obtain

\begin{equation}\label{n7}
\|u\|_{L^2(\Omega)}=\frac{O(1)}{\la}\quad\text{and}\quad \|y\|_{L^2(\Omega)}=\frac{O(1)}{\la}.
\end{equation}

\begin{Lemma}\label{Lemma1h3h4}
Assume that the assumption (H2) or (H3) holds. We have that the solution $(u, v, y, z) \in  D(\mathcal{A})$ of \eqref{n3}-\eqref{n6} satisfies the following estimations
\begin{equation}\label{eq1-lem1-pol2}
\|\nabla v\|_{L^2(\omega_b)}=\frac{o(1)}{\la}\quad \text{and}\quad\|\nabla u\|_{L^2(\omega_b)}=\frac{o(1)}{\la^{2}}.
\end{equation}
\end{Lemma}
\noindent The proof of this Lemma is similar to that of Lemma \ref{lem1-pol}.
\begin{Lemma}\label{lemma2h3}
Under the assumption (H2). We have that the solution $(u, v, y, z) \in  D(\mathcal{A})$ of \eqref{n3}-\eqref{n6} satisfies the following estimations
\begin{equation}\label{eq2-lem1-pol2}
 \|u\|_{L^2(\omega_b)}=\frac{o(1)}{\la^{2}}.
\end{equation}
\end{Lemma}
\begin{proof}
By using Poincaré inequality and Equation \eqref{eq1-lem1-pol2}, we get
\begin{equation}\label{eq01-lem1-pol2}
\| v\|_{L^2(\omega_b)}=\frac{o(1)}{\la}.
\end{equation}
From Equation \eqref{n3} and \eqref{eq01-lem1-pol2}, we obtain
\begin{equation}
 \|u\|_{L^2(\omega_b)}=\frac{o(1)}{\la^{2}}.
\end{equation}
\end{proof}

\noindent Inserting Equations \eqref{n3} and \eqref{n5} into \eqref{n4} and \eqref{n6}, we get
\begin{eqnarray}
\lam ^{2}u+\divv(a\nabla u+b(x)\nabla v)-i\la c(x)y&=&-\frac{f_2}{\la^{2}}-c(x)\frac{f_3}{\la^{2}}-\frac{if_1}{\la},\label{n8}\\
\la ^{2}y+\Delta y+i\la c(x)u&=&-\frac{f_4}{\la^{2}}+c(x)\frac{f_1}{\la^{2}}-\frac{if_3}{\la}.\label{n9}
\end{eqnarray}

\begin{Lemma}\label{lem2-pol3}
Assume that assumption (H3) holds. Then, the solution $(u, v, y, z) \in  D(\mathcal{A})$ of \eqref{n3}-\eqref{n6} satisfies the following estimation
\begin{equation}\label{1eq1-lem2-pol3}
\int_{\tilde{\omega}_b}\abs{\la u}^2dx=o(1).
\end{equation}
such that $\omega_c\subset\tilde{\omega}_b\subset\omega_b$.
\end{Lemma}

\begin{proof}
Let a non-empty open subset  $\tilde{\omega}_b$   such that $\omega_c\subset\tilde{\omega}_b\subset\omega_b$. Then, we define the function $h_1\in C_{c}^{\infty}(\R^N)$ such that
\begin{equation}\label{h1}
h_1(x)=\left\{\begin{array}{ccc}
1&\text{if}&x\in \tilde{\omega}_b,\\
0&\text{if}& x\in \Omega\backslash\omega_b,\\
0\leq h_1\leq 1&& elsewhere.
\end{array}\right.
\end{equation}
Multiply \eqref{n8} by $h_1\overline{u}$ and integrate over $\Omega$, we get
\begin{equation}\label{Est0-lem2}
\int_{\Omega}h_1\abs{\la u}^2dx-\int _{\Omega}(a\nabla u+b(x)\nabla v)\cdot(h_1\nabla\overline{u}+\nabla h_1\overline{u})dx-i\la \int_{\Omega}h_1c(x)y\overline{u}dx=\frac{o(1)}{\la^2}.
\end{equation}
Using \eqref{n7} and \eqref{eq1-lem1-pol2}, we have
\begin{equation}\label{Est1-lem2}
\left|\int _{\Omega}(a\nabla u+b(x)\nabla v)\cdot(h_1\nabla\overline{u}+\nabla h_1\overline{u})dx\right|=\frac{o(1)}{\la^{2}}
\end{equation}
and 
\begin{equation}\label{Est2-lem2}
\left|i\la\int_{\Omega}h_1c(x)y\overline{u}dx\right|=\frac{O(1)}{\la}.
\end{equation}
Thus, by using Equations \eqref{Est1-lem2} and \eqref{Est2-lem2} in \eqref{Est0-lem2}, we obtain
\begin{equation}
\int_{\Omega}h_1\abs{\la u}^2dx=\frac{O(1)}{\la}.
\end{equation}
Thus, we reach our desired result.
\end{proof}

\begin{Lemma}\label{lem2-pol2}
Assume that assumption (H2) or (H3) holds. Then, the solution $(u, v, y, z) \in  D(\mathcal{A})$ of \eqref{n3}-\eqref{n6} satisfies the following estimation
\begin{equation}\label{eq1-lem2-pol2}
\int_{\omega_c}\abs{\la y}^2dx=o(1).
\end{equation}
\end{Lemma}

\begin{proof}
\textbf{Case1.}
Assume that assumption (H2) holds, we define the function $\rho\in C_{c}^{\infty}(\R^N)$ such that
\begin{equation}\label{rho}
\rho(x)=\left\{\begin{array}{ccc}
1&\text{if}&x\in \omega_{c},\\
0&\text{if}& x\in \Omega\backslash\omega_b,\\
0\leq\rho\leq 1&& elsewhere.
\end{array}\right.
\end{equation}
Now, multiplying \eqref{n8} by $\la \rho \overline{y}$, integrating over $\Omega$ and using Green's formula, \eqref{n7} and the fact that $\|F\|_{\mathcal{H}}=\|(f_1,f_2,f_3,f_4)\|_{\mathcal{H}}=o(1)$, we get

\begin{equation}\label{est1-lem2}
\la^3\int_{\Omega}u\rho\overline{y}dx-\la\int_{\Omega}\left(a\nabla u+b(x)\nabla v\right)\cdot(\nabla\rho \overline{y}+\rho\nabla\overline{y})dx-i\int_{\Omega}c(x)\rho\abs{\la y}^2dx=\frac{o(1)}{\la}.
\end{equation}
Using \eqref{n7},  \eqref{eq1-lem1-pol2} and Cauchy-Schwarz we obtain
\begin{equation}\label{est2-lem2}
\left|\la^3\int_{\Omega}u\rho\overline{y}dx\right|\leq \la^3 \|u\|_{L^2(\omega_b)}\cdot\|y\|_{L^2(\omega_b)}=o(1).
\end{equation}
and 
\begin{equation}\label{est3-lem2}
\left|\la\int_{\Omega}\left(a\nabla u+b(x)\nabla v\right)\cdot(\nabla\rho \overline{y}+\rho\nabla\overline{y})dx\right|=o(1).
\end{equation}
Thus, using Equations \eqref{est2-lem2} and \eqref{est3-lem2} in \eqref{est1-lem2} we obtain our desired result for the first case.\\
\textbf{Case 2.}
Assume  that assumption (H3) holds.  Define the function $h_2\in C_{c}^{\infty}(\R^N)$ such that
\begin{equation}\label{h1}
h_2(x)=\left\{\begin{array}{ccc}
1&\text{if}&x\in \omega_{c},\\
0&\text{if}& x\in \Omega\backslash\tilde{\omega}_b,\\
0\leq h_2\leq 1&& elsewhere.
\end{array}\right.
\end{equation}
Multiply \eqref{n8} by $\la h_2\overline{y}$ and integrate over $\Omega$, and using Green's formula, \eqref{n7} and the fact that $\|F\|_{\mathcal{H}}=\|(f_1,f_2,f_3,f_4)\|_{\mathcal{H}}=o(1)$, we get

\begin{equation}\label{Est1-lem3}
\la^3\int_{\Omega}uh_2\overline{y}dx-\la\int_{\Omega}\left(a\nabla u+b(x)\nabla v\right)\cdot(\nabla h_2 \overline{y}+h_2\nabla\overline{y})dx-i\int_{\Omega}c(x)h_2\abs{\la y}^2dx=\frac{o(1)}{\la}.
\end{equation}
Using \eqref{n7},  \eqref{eq1-lem1-pol2} and Cauchy-Schwarz,  we obtain
\begin{equation}\label{new1}
\left|\la\int_{\Omega}\left(a\nabla u+b(x)\nabla v\right)\cdot(\nabla h_2 \overline{y}+h_2\nabla\overline{y})dx\right|=o(1).
\end{equation}
By using Equation \eqref{new1} in \eqref{Est1-lem3}, we obtain
\begin{equation}\label{new2}
\la^3\int_{\Omega}uh_2\overline{y}dx-i\int_{\Omega}c(x)h_2\abs{\la y}^2dx=o(1).
\end{equation}
Now, multiply \eqref{n9} by $\la h_2\overline{u}$ and integrate over $\Omega$, and using Green's formula, \eqref{n7} and the fact that $\|F\|_{\mathcal{H}}=\|(f_1,f_2,f_3,f_4)\|_{\mathcal{H}}=o(1)$, we get
\begin{equation}\label{new3}
\la^3\int_{\Omega}yh_2\overline{u}dx-\la\int_{\Omega}\nabla y(\nabla h_2 \overline{u}+h_2\nabla\overline{u})dx+ic_0\int_{\omega_c}\abs{\la u}^2dx=\frac{o(1)}{\la}.
\end{equation}
Using \eqref{n1}, \eqref{n7} and Cauchy-Schwarz,  we obtain
\begin{equation}\label{new4}
\left|\la\int_{\Omega}\nabla y\cdot(\nabla h_2 \overline{u}+h_2\nabla\overline{u})dx\right|=o(1).
\end{equation}
Inserting \eqref{new4} into \eqref{new3}, we get
\begin{equation}\label{new5}
\la^3\int_{\Omega}yh_2\overline{u}dx+ic_0\int_{\omega_c}\abs{\la u}^2dx=o(1).\end{equation}
Summing Equations \eqref{new2} and \eqref{new5}, and taking the imaginary part and using Equation \eqref{1eq1-lem2-pol3}, we obtain our desired result.
\end{proof}

\begin{Lemma}\label{lem3-pol2}
Assume that assumption (H2) or (H3) holds. Then,  the solution $(u, v, y, z) \in  D(\mathcal{A})$ of \eqref{n3}-\eqref{n6} satisfies the following estimations
\begin{equation}
\int_{\Omega}\left|\la u\right|^2dx=o(1)\quad \text{and}\quad\int_{\Omega}\left|\la y\right|^2dx=o(1).
\end{equation}
\end{Lemma}

\begin{proof}
 Let $\varphi, \psi \in H^2(\Omega)\cap H_0^1(\Omega)$ be the solution of the following system
\begin{equation}\label{aux4}
\left\{
\begin{array}{lll}
\lam ^{2}\varphi+a\Delta\varphi-i\lam\mathds{1}_{\omega_c}(x)\varphi=u,&\text{in} &\Omega\\ 
\lam ^{2}\psi+\Delta\psi-i\lam\mathds{1}_{\omega_c}(x)\psi=y, &\text{in}& \Omega\\ 
\varphi=\psi=0,&\text{on}&\Gamma\\ 
\end{array}\right.
\end{equation}
where $(u,v,y,z)$ is the solution of \eqref{n3}-\eqref{n6}.  Since either (H2) or (H3) holds, then system \eqref{AUXXX} is exponentially stable. Thus, there exists $M>0$ such that system \eqref{aux4} satisfies the following estimation
\begin{equation}\label{bdd1}
\|\lam\varphi\|_{L^{2}(\Omega)}+\|\nabla\varphi\|_{L^{2}(\Omega)}+\|\lam\psi\|_{L^{2}(\Omega)}+\|\nabla\psi\|_{L^{2}(\Omega)}\leq M\left(\| u\|_{L^{2}(\Omega)}+\| y\|_{L^{2}(\Omega)}\right).
\end{equation}
\textbf{Case 1.} Under the assumption (H3). Multiply \eqref{n8} by $\la^2\overline{\varphi}$ and integrate over $\Omega$, and using Green's formula, Equation \eqref{bdd1}, and the fact that $\|F\|_{\mathcal{H}}=\|(f_1,f_2,f_3,f_4)\|_{\mathcal{H}}=o(1)$, we obtain
\begin{equation}\label{est6-1}
\int_{\Omega}\left(\la^2\overline{\varphi}+a\Delta\overline{ \varphi}\right)\la^2 u dx-\la^{2}\int_{\Omega}b(x)\nabla v\nabla\bar{\varphi}dx-\int_{\Omega} i\la^3 c(x)y\overline{\varphi}dx=\frac{o(1)}{\la}.
\end{equation}
From Equation \eqref{eq1-lem1-pol2} and \eqref{bdd1}, we obtain 
\begin{equation}\label{est7-1}
\left|\la^{2}\int_{\Omega}b(x)\nabla v\cdot\nabla\bar{\varphi}dx\right|=o(1).
\end{equation}
Now, using System \eqref{aux4} and Equation \eqref{est7-1} in \eqref{est6-1}, we get
\begin{equation}\label{est10-1}
\int_{\Omega}\left|\la u\right|^2dx-i\la^3\int_{\Omega}\mathds{1}_{\omega_{c}}(x)u\overline{\varphi}dx-i\la^3 \int_{\Omega} c(x)y\overline{\varphi}dx=o(1).
\end{equation}
By using \eqref{eq2-lem1-pol2}, \eqref{bdd1} and the fact that $\omega_c\subset\omega_b$
\begin{equation}\label{est11-1}
\left|i\la^3\int_{\Omega}\mathds{1}_{\omega_c}(x)u\overline{\varphi}dx\right|\leq \la^3 \|u\|_{L^2(\omega_c)}\cdot\|\varphi\|_{L^2(\Omega)}=\frac{o(1)}{\la}.
\end{equation}
Using \eqref{eq1-lem2-pol2} and \eqref{bdd1}, we get that
\begin{equation}\label{est12-1}
\left|i\la^3\int_{\Omega}  c(x)y\overline{\varphi}dx\right|=o(1).
\end{equation}
Now, inserting Equations \eqref{est11-1} and \eqref{est12-1} into \eqref{est10-1} , we get
\begin{equation}\label{est12}
\int_{\Omega}\left|\la u\right|^2dx=o(1).
\end{equation}
Multiply \eqref{n9} by $\la^2\overline{\psi}$ and integrate over $\Omega$, and using Green's formula, Equation \eqref{bdd1}, and the fact that $\|F\|_{\mathcal{H}}=\|(f_1,f_2,f_3,f_4)\|_{\mathcal{H}}=o(1)$, we obtain
\begin{equation}\label{eq1}
\int_{\Omega}\left(\la^2\overline{\psi}+\Delta\overline{ \psi}\right)\la^2 y dx+i\la^3\int_{\Omega}  c(x)u\overline{\psi}dx=\frac{o(1)}{\la}.
\end{equation}
 By using System \eqref{aux4} in \eqref{eq1}, we get
 \begin{equation}\label{eq2}
\int_{\Omega}\left|\la y\right|^2dx-i\la^3\int_{\Omega}\mathds{1}_{\omega_{c}}(x)y\overline{\psi}dx+i\la^3\int_{\Omega}  c(x)u\overline{\psi}dx=o(1). \end{equation}
Using \eqref{eq2-lem1-pol2}, \eqref{eq1-lem2-pol2}, and \eqref{bdd1} and the fact that $\omega_c\subset\omega_b$, we get
\begin{equation}\label{eq3}
\left|i\la^3\int_{\Omega}\mathds{1}_{\omega_{c}}(x)y\overline{\psi}dx\right|=o(1).
\end{equation}
and 
\begin{equation}\label{eq4}
\left| i\la^3\int_{\Omega}  c(x)u\overline{\psi}dx\right|=\frac{o(1)}{\la}.
\end{equation}
Inserting \eqref{eq3} and \eqref{eq4} into \eqref{eq2} we obatin 
\begin{equation}
\int_{\Omega}\left|\la y\right|^2dx=o(1).
\end{equation}
\textbf{Case 2.} Under the assumption (H3). We proceed in the same way as in Case 1., the only change is that we have the following two estimations instead of \eqref{est11-1} and \eqref{eq4}, by using \eqref{1eq1-lem2-pol3} and \eqref{bdd1} we get
\begin{equation}
\left|i\la^3\int_{\Omega}\mathds{1}_{\omega_c}(x)u\overline{\varphi}dx\right|=o(1)
\end{equation}
and
\begin{equation}
\left| i\la^3\int_{\Omega}  c(x)u\overline{\psi}dx\right|=o(1).
\end{equation}
Thus, we reach our desired result.
\end{proof}
\begin{Lemma}\label{lem4-pol2}
Assume that either the assumption (H2) or (H3) holds. Then,  the solution $(u, v, y, z) \in  D(\mathcal{A})$ of \eqref{n3}-\eqref{n6} satisfies the following estimations

\begin{equation}
\int_{\Omega}\left|\nabla u\right|^2dx=o(1)\quad \text{and}\quad\int_{\Omega}\left|\nabla y\right|^2dx=o(1).
\end{equation}

\end{Lemma}
\begin{proof}
The proof is similar to the proof of Lemma \ref{lem4-pol}.
\end{proof}

\noindent \textbf{Proof of Theorem \ref{Theorem2}.} 
Consequently, from the results of Lemmas \ref{lem3-pol2}, and \ref{lem4-pol2} , we obtain  that $\|U\|_{\HH}=o(1)$, which contradicts \eqref{n1}. Consequently, condition \eqref{C3} holds. This implies, from Theorem \ref{bt}, the energy decay estimation \eqref{Energyp2}. The proof is thus complete.

\begin{Theorem}\label{Theorem3}
Assume that assumption (H4) or (H5) holds. Then, for all initial data $U_0\in D(\mathcal{A})$, there exists a constant $C_3>0$ independent of $U_0$, the energy of the system \eqref{k-v}-\eqref{boundary conditions} satisfies the following estimation
\begin{equation}\label{Energyp3}
	E(t,U)\leq \frac{C_3}{t^{\frac{2}{2+4\beta}}}\|U_0\|^2_{D(\mathcal{A})},\quad \forall t>0.
	\end{equation}
	where
	\begin{equation}\label{bbb}
\beta=\left\{\begin{array}{ccc}
2&\text{if}&\text{(H4)} \,\text{holds}\\ \\
\dfrac{3}{2}&\text{if}&\text{(H5)} \,\text{holds}.
\end{array}
\right.
\end{equation}
\end{Theorem}
\noindent Following Theorem \ref{bt} of Borichev and Tomilov (see Appendix), the polynomial energy decay \eqref{Energyp3} holds if \eqref{C1} and 
\begin{equation}\label{C4}\tag{${\rm{C4}}$}
\sup_{\la\in \R}\left\|(i\la I-\AA)^{-1}\right\|_{\mathcal{L}(\mathcal{H})}=O\left(\abs{\la}^{{2+4\beta}}\right)
\end{equation}
are satisfied. Since Condition \eqref{C1} is already satisfied (see Lemmas  \ref{ker1} and \ref{ker2}). We will prove condition \eqref{C4} by an argument of contradiction. For this purpose, suppose that \eqref{C4} is false,  then there exists 
$\left\{\left(\la_n,U_n:=(u_n,v_n,y_n,z_n)^\top\right)\right\}\subset \R^{\ast}\times D(\AA)$ with 
\begin{equation}\label{m1}
\abs{\la_n}\to +\infty \quad \text{and}\quad \|U_n\|_{\mathcal{H}}=\|(u_n,v_n,y_n,z_n)\|_{\mathcal{H}}=1, 
\end{equation}
such that 
\begin{equation}\label{m2}
\la_n^{2+4\beta}\left(i\la_nI-\AA\right)U_n=F_n:=(f_{1,n},f_{2,n},f_{3,n},f_{4,n})^{\top}\to 0 \ \ \text{in}\ \ \mathcal{H}. 
\end{equation}
For simplicity, we drop the index $n$. Equivalently, from \eqref{m2}, we have
\begin{eqnarray}
i\la u-v&=&\frac{f_1}{\la^{2+4\beta}} \ \text{in}\ \ H_0^1(\Omega),\label{m3}\\ \noalign{\medskip}
i\la v-\divv(a \nabla u+b(x)\nabla v)+c(x)z&=&\frac{f_2}{\la^{2+4\beta}}\ \ \text{in}\ \ L^2(\Omega),\label{m4}\\ \noalign{\medskip}
i\la y-z&=&\frac{f_3}{\la^{2+4\beta}}\ \ \text{in}\ \ H_0^1(\Omega),\label{m5}\\ \noalign{\medskip}
i\la z- \divv(\nabla y)-c(x)v&=&\frac{f_4}{\la^{2+4\beta}}\ \ \text{in}\ \ L^2(\Omega).\label{m6}
\end{eqnarray}
Here we will check the condition \eqref{C4} by finding a contradiction with \eqref{m1} such as $\left\|U\right\|_{\HH}=o(1)$. 
From Equations \eqref{m1}, \eqref{m3} and \eqref{m5},  we obtain

\begin{equation}\label{m7}
\|u\|_{L^2(\Omega)}=\frac{O(1)}{\la}\quad\text{and}\quad \|y\|_{L^2(\Omega)}=\frac{O(1)}{\la}.
\end{equation}

\begin{Lemma}\label{Lemma1h5h6}
Under the assumptions (H4) or (H5). We have that the solution $(u, v, y, z) \in  D(\mathcal{A})$ of \eqref{m3}-\eqref{m6} satisfies the following estimations
\begin{equation}\label{eq1-lem1-pol3}
\|\nabla v\|_{L^2(\omega_b)}=\frac{o(1)}{\la^{1+2\beta}}\quad \|\nabla u\|_{L^2(\omega_b)}=\frac{o(1)}{\la^{2+2\beta}}\quad \text{and}\quad  \|u\|_{L^2(\omega_b)}=\frac{o(1)}{\la^{2+2\beta}}.
\end{equation}
\end{Lemma}
\noindent The proof of this Lemma is similar to that of Lemma \ref{lem1-pol}.

\noindent Inserting Equations \eqref{m3} and \eqref{m5} into \eqref{m4} and \eqref{m6}, we get
\begin{eqnarray}
\lam ^{2}u+\divv(a\nabla u+b(x)\nabla v)-i\la c(x)y&=&-\frac{f_2}{\la^{2+4\beta}}-c(x)\frac{f_3}{\la^{2+4\beta}}-\frac{if_1}{\la^{1+4\beta}},\label{m8}\\
\la ^{2}y+\Delta y+i\la c(x)u&=&-\frac{f_4}{\la^{2+4\beta}}+c(x)\frac{f_1}{\la^{2+4\beta}}-\frac{if_3}{\la^{1+4\beta}}.\label{m9}
\end{eqnarray}

\begin{Lemma}\label{lem2-pol3}
Assume that assumption (H4) or (H5) holds. Then, the solution $(u, v, y, z) \in  D(\mathcal{A})$ of \eqref{m3}-\eqref{m6} satisfies the following estimation
\begin{equation}\label{eq1-lem2-pol3}
\int_{\omega_c}\abs{\la y}^2dx=\frac{o(1)}{\la^{2\beta}}.
\end{equation}
\end{Lemma}

\begin{proof}
Assume that either assumption (H4) or (H5) holds. Define the function $\rho\in C_{c}^{\infty}(\R^N)$ such that
\begin{equation}\label{rho}
\rho(x)=\left\{\begin{array}{ccc}
1&\text{if}&x\in \omega_{c},\\
0&\text{if}& x\in \Omega\backslash\omega_b,\\
0\leq\rho\leq 1&& elsewhere.
\end{array}\right.
\end{equation}
Now, multiply \eqref{m8} by $\la \rho \overline{y}$, integrate over $\Omega$ and using Green's formula, \eqref{m7} and the fact that $\|F\|_{\mathcal{H}}=\|(f_1,f_2,f_3,f_4)\|_{\mathcal{H}}=o(1)$, we get

\begin{equation}\label{est1-lem2p}
\la^3\int_{\Omega}u\rho\overline{y}dx-\la\int_{\Omega}\left(a\nabla u+b(x)\nabla v\right)\cdot(\nabla\rho \overline{y}+\rho\nabla\overline{y})dx-i\int_{\Omega}c(x)\rho\abs{\la y}^2dx=\frac{o(1)}{\la^{1+4\beta}}.
\end{equation}
Using \eqref{m7},  \eqref{eq1-lem1-pol3} and Cauchy-Schwarz we obtain
\begin{equation}\label{est2-lem2p}
\left|\la^3\int_{\Omega}u\rho\overline{y}dx\right|\leq \la^3 \|u\|_{L^2(\omega_b)}\cdot\|y\|_{L^2(\omega_b)}=\frac{o(1)}{\la^{2\beta}}
\end{equation}
and 
\begin{equation}\label{est3-lem2p}
\left|\la\int_{\Omega}\left(a\nabla u+b(x)\nabla v\right)\cdot(\nabla\rho \overline{y}+\rho\nabla\overline{y})dx\right|=\frac{o(1)}{\la^{2\beta}}.
\end{equation}
Thus, using Equation \eqref{est2-lem2p} and \eqref{est3-lem2p} in \eqref{est1-lem2p} we obtain our desired result.
\end{proof}

\begin{Lemma}\label{lem3-pol3}
Assume that the assumption (H4) or (H5) holds. Then,  the solution $(u, v, y, z) \in  D(\mathcal{A})$ of \eqref{m3}-\eqref{m6} satisfies the following estimations
\begin{equation}
\int_{\Omega}\left|\la u\right|^2dx=o(1)\quad \text{and}\quad\int_{\Omega}\left|\la y\right|^2dx=o(1).
\end{equation}
\end{Lemma}

\begin{proof}
Let $\varphi, \psi \in H^2(\Omega)\cap H_0^1(\Omega)$ be the solution of the following system
\begin{equation}\label{aux4p}
\left\{
\begin{array}{lll}
\lam ^{2}\varphi+a\Delta\varphi-i\lam\mathds{1}_{\omega_c}(x)\varphi=u,&\text{in} &\Omega\\ 
\lam ^{2}\psi+\Delta\psi-i\lam\mathds{1}_{\omega_c}(x)\psi=y, &\text{in}& \Omega\\ 
\varphi=\psi=0,&\text{on}&\Gamma\\ 
\end{array}\right.
\end{equation}
where $(u,v,y,z)$ is the solution of \eqref{m3}-\eqref{m6}. 
We suppose that the energy of the System  \eqref{AUXXX} satisfies the following estimate
$$E(t,U)\leq \frac{C}{t^{\frac{2}{\beta}}}\|U_0\|^2_{D(\mathcal{A})},\quad \forall t>0.$$
When assumption (H4) holds, we have that System \eqref{AUXXX} is polynomially stable with an energy decay rate $t^{-1}$, i.e $\beta=2$. Howerver, when assumption (H5) holds then we have that System \eqref{AUXXX} is polynomially stable with an energy decay rate $t^{-\frac{4}{3}}$, i.e $\beta=\frac{3}{2}$. Thus,  there exists $M>0$ such that system \eqref{aux4p} satisfies the following estimation
\begin{equation}\label{bdd1p}
\|\lam\varphi\|_{L^{2}(\Omega)}+\|\nabla\varphi\|_{L^{2}(\Omega)}+\|\lam\psi\|_{L^{2}(\Omega)}+\|\nabla\psi\|_{L^{2}(\Omega)}\leq M\abs{\la}^{\beta}\left(\| u\|_{L^{2}(\Omega)}+\| y\|_{L^{2}(\Omega)}\right).
\end{equation}
Assuming that (H4) or (H5) holds.  Multiply \eqref{m8} by $\la^2\overline{\varphi}$ and integrate over $\Omega$, and using Green's formula, Equation \eqref{bdd1}, and the fact that $\|F\|_{\mathcal{H}}=\|(f_1,f_2,f_3,f_4)\|_{\mathcal{H}}=o(1)$, we obtain
\begin{equation}\label{est6-1p}
\int_{\Omega}\left(\la^2\overline{\varphi}+a\Delta\overline{ \varphi}\right)\la^2 u dx-\la^{2}\int_{\Omega}b(x)\nabla v\cdot\nabla\overline{\varphi}dx-\int_{\Omega} i\la^3 c(x)y\overline{\varphi}dx=\frac{o(1)}{\la^{1+3\beta}}.
\end{equation}
From Equation \eqref{eq1-lem1-pol3} and \eqref{bdd1p}, we obtain 
\begin{equation}\label{est7-1p}
\left|\la^{2}\int_{\Omega}b(x)\nabla v\nabla\overline{\varphi}dx\right|=\frac{o(1)}{\la^{\beta}}.
\end{equation}
Now, using System \eqref{aux4p} and Equation \eqref{est7-1p} in \eqref{est6-1p}, we get
\begin{equation}\label{est10-1p}
\int_{\Omega}\left|\la u\right|^2dx-i\la^3\int_{\Omega}\mathds{1}_{\omega_{c}}(x)u\overline{\varphi}dx-i\la^3 \int_{\Omega} c(x)y\overline{\varphi}dx=\frac{o(1)}{\la^{\beta}}.
\end{equation}
By using \eqref{eq1-lem1-pol3}, \eqref{bdd1p} and the fact that $\omega_c\subset\omega_b$
\begin{equation}\label{est11-1p}
\left|i\la^3\int_{\Omega}\mathds{1}_{\omega_c}(x)u\overline{\varphi}dx\right|\leq \la^3 \|u\|_{L^2(\omega_c)}\cdot\|\varphi\|_{L^2(\Omega)}=\frac{o(1)}{\la^{1+\beta}}.
\end{equation}
Using \eqref{eq1-lem2-pol3} and \eqref{bdd1p}, we get that
\begin{equation}\label{est12-1p}
\left|i\la^3\int_{\Omega}  c(x)y\overline{\varphi}dx\right|=o(1).
\end{equation}
Now, inserting Equation \eqref{est11-1p} and \eqref{est12-1p} into \eqref{est10-1p},  we get
\begin{equation}\label{est12p}
\int_{\Omega}\left|\la u\right|^2dx=o(1).
\end{equation}
Multiply \eqref{m9} by $\la^2\overline{\psi}$ and integrate over $\Omega$, and using Green's formula, Equation \eqref{bdd1p}, and the fact that $\|F\|_{\mathcal{H}}=\|(f_1,f_2,f_3,f_4)\|_{\mathcal{H}}=o(1)$, we obtain
\begin{equation}\label{eq1p}
\int_{\Omega}\left(\la^2\overline{\psi}+\Delta\overline{ \psi}\right)\la^2 y dx+i\la^3\int_{\Omega}  c(x)u\overline{\psi}dx=\frac{o(1)}{\la^{1+3\beta}}.
\end{equation}
 By using System \eqref{aux4p} in \eqref{eq1p} we get
 \begin{equation}\label{eq2p}
\int_{\Omega}\left|\la y\right|^2dx-i\la^3\int_{\Omega}\mathds{1}_{\omega_{c}}(x)y\overline{\psi}dx+i\la^3\int_{\Omega}  c(x)u\overline{\psi}dx=\frac{o(1)}{\la^{1+3\beta}}. \end{equation}
Using \eqref{eq1-lem1-pol3}, \eqref{eq1-lem2-pol3}, and \eqref{bdd1p} and the fact that $\omega_c\subset\omega_b$, we get
\begin{equation}\label{eq3p}
\left|i\la^3\int_{\Omega}\mathds{1}_{\omega_{c}}(x)y\overline{\psi}dx\right|=o(1).
\end{equation}
and 
\begin{equation}\label{eq4p}
\left| i\la^3\int_{\Omega}  c(x)u\overline{\psi}dx\right|=\frac{o(1)}{\la^{1+\beta}}.
\end{equation}
Inserting \eqref{eq3p} and \eqref{eq4p} into \eqref{eq2p} we obtain 
\begin{equation}
\int_{\Omega}\left|\la y\right|^2dx=o(1).
\end{equation}
\end{proof}
\begin{Lemma}\label{lem4-pol3}
Assume that either the assumption (H4) or (H5) holds. Then,  the solution $(u, v, y, z) \in  D(\mathcal{A})$ of \eqref{m3}-\eqref{m6} satisfies the following estimations
\begin{equation}
\int_{\Omega}\left|\nabla u\right|^2dx=o(1)\quad \text{and}\quad\int_{\Omega}\left|\nabla y\right|^2dx=o(1).
\end{equation}

\end{Lemma}
\begin{proof}
The proof is similar to the proof of Lemma \ref{lem4-pol}.
\end{proof}

\noindent \textbf{Proof of Theorem \ref{Theorem3}.} 
Consequently, from the results of Lemma \ref{lem3-pol3}, \ref{lem4-pol3} , we obtain  that $\|U\|_{\HH}=o(1)$, which contradicts \eqref{m1}. Consequently, condition \eqref{C4} holds. This implies, from Theorem \ref{bt}, the energy decay estimation \eqref{Energyp3}. The proof is thus complete.

\section{Appendix}
\noindent In this section, we introduce the notions of stability that we encounter in this work.


\begin{definition}\label{Defsta}
{Assume that $A$ is the generator of a C$_0$-semigroup of contractions $\left(e^{tA}\right)_{t\geq0}$  on a Hilbert space  $\mathcal{H}$. The  $C_0$-semigroup $\left(e^{tA}\right)_{t\geq0}$  is said to be
\begin{enumerate}
\item[1.]  strongly stable if 
$$\lim_{t\to +\infty} \|e^{tA}x_0\|_{H}=0, \quad\forall \ x_0\in H;$$
\item[2.]  exponentially (or uniformly) stable if there exist two positive constants $M$ and $\epsilon$ such that
\begin{equation*}
\|e^{tA}x_0\|_{H} \leq Me^{-\epsilon t}\|x_0\|_{H}, \quad
\forall\  t>0,  \ \forall \ x_0\in {H};
\end{equation*}
\item[3.] polynomially stable if there exists two positive constants $C$ and $\alpha$ such that
\begin{equation*}
 \|e^{tA}x_0\|_{H}\leq C t^{-\alpha}\|x_0\|_{H},  \quad\forall\ 
t>0,  \ \forall \ x_0\in D\left(\mathcal{A}\right).
\end{equation*}
In that case, one says that the semigroup $\left(e^{tA}\right)_{t\geq 0}$ decays  at a rate $t^{-\alpha}$.
\noindent The  $C_0$-semigroup $\left(e^{tA}\right)_{t\geq0}$  is said to be  polynomially stable with optimal decay rate $t^{-\alpha}$ (with $\alpha>0$) if it is polynomially stable with decay rate $t^{-\alpha}$ and, for any $\varepsilon>0$ small enough, the semigroup $\left(e^{tA}\right)_{t\geq0}$  does  not decay at a rate $t^{-(\alpha-\varepsilon)}$.
\end{enumerate}}
\end{definition}
\noindent To show the strong stability of a $C_0-$semigroup of contraction  $(e^{tA})_{t\geq 0}$ we rely on the following result due to Arendt-Batty \cite{Arendt01}.

\begin{theoreme}\label{arendtbatty}
Assume that $A$ is the generator of a C$_0-$semigroup of contractions $\left(e^{tA}\right)_{t\geq0}$  on a Hilbert space $\mathcal{H}$. If
 \begin{enumerate}
 \item[1.]  $A$ has no pure imaginary eigenvalues,
  \item[2.]  $\sigma\left(A\right)\cap i\mathbb{R}$ is countable,
 \end{enumerate}
where $\sigma\left(A\right)$ denotes the spectrum of $A$, then the $C_0-$semigroup $\left(e^{tA}\right)_{t\geq0}$  is strongly stable.
\end{theoreme}

\noindent  Concerning the characterization of exponential stability of a $C_0-$semigroup of contraction  $(e^{tA})_{t\geq 0}$ we rely on the following result due to Huang \cite{Huang01} and Pr\"uss \cite{pruss01}. 
\begin{theoreme}\label{hp}
Let $A:\ D(A)\subset H\rightarrow H$ generate a $C_0-$semigroup of contractions $\left(e^{tA}\right)_{t\geq 0}$ on $H$. Assume that $i\la \in \rho(A)$, $\forall \la \in \R$. Then, the $C_0-$semigroup $\left(e^{tA}\right)_{t\geq 0}$ is exponentially stable if and only if 
$$
\varlimsup_{\la\in\R,\ \abs{\la}\to +\infty}\|(i\la I-A)^{-1}\|_{\mathcal{L}(H)}<+\infty.
$$
\end{theoreme}
\noindent  Also, concerning the characterization of polynomial stability of a $C_0-$semigroup of contraction  $(e^{tA})_{t\geq 0}$ we rely on the following result due to Borichev and Tomilov \cite{Borichev01} (see also \cite{RaoLiu01} and \cite{Batty01}). 

\begin{theoreme}\label{bt}
Assume that $A$ is the generator of a strongly continuous semigroup of contractions $\left(e^{tA}\right)_{t\geq0}$  on $H$.   If   $ i\mathbb{R}\subset \rho(A)$, then for a fixed $\ell>0$ the following conditions are equivalent
\begin{equation}\label{h1}
\sup_{\lambda\in\mathbb{R}}\left\|\left(i\lambda I-A\right)^{-1}\right\|_{\mathcal{L}\left(H\right)}=O\left(|\lambda|^\ell\right),
\end{equation}
\begin{equation}\label{h2}
\|e^{tA}U_{0}\|^2_{H} \leq \frac{C}{t^{\frac{2}{\ell}}}\|U_0\|^2_{D(A)},\hspace{0.1cm}\forall t>0,\hspace{0.1cm} U_0\in D(A),\hspace{0.1cm} \text{for some}\hspace{0.1cm} C>0.
\end{equation}
\end{theoreme}


\end{document}